\documentclass[10pt]{amsart}
\usepackage{amsfonts}
\usepackage{amssymb}
\usepackage{amsmath}
\usepackage{amsthm}

\newcommand{\te}{Teich\-m\"ul\-ler}
\newcommand{\tes}{Teich\-m\"ul\-ler's}

\newtheorem{theorem}{Theorem}

\newtheorem{lemma}{Lemma}
\newtheorem{proposition}{Proposition}

\theoremstyle{definition}
\newtheorem{definition}{Definition}

\newcommand{\C}{\overline{\mathbb C}}
\newcommand{\CP}{\mathbb C}

\begin{document}

\title[]{Holomorphic Motions and Related Topics}

\author{Frederick Gardiner, Yunping Jiang, and Zhe Wang}

\thanks{The first and the second authors are supported by PSC-CUNY
awards}

%\date{January, 2008}
\begin{abstract}
In this article we give an expository account of the holomorphic
motion theorem based on work of M\~ane-Sad-Sullivan, Bers-Royden,
and Chirka. After proving this theorem, we show that tangent
vectors to holomorphic motions have  $|\epsilon \log \epsilon|$
moduli of continuity and then show how this type of continuity for
tangent vectors can be combined with Schwarz's lemma and
integration over the holomorphic variable to produce H\"older
continuity on the mappings. We also prove, by using holomorphic
motions, that Kobayashi's and Teichm\"uller's metrics on the
Teichm\"uller space of a Riemann surface coincide. Finally, we
present an application of holomorphic motions to complex dynamics,
that is, we prove the Fatou linearization theorem for parabolic
germs by involving holomorphic motions.
\end{abstract}

\maketitle

\section{Introduction}

Suppose $\C={\mathbb C} \cup \{\infty\}$ is the extended complex
plane and $E\subset \C$ is a subset. For any real number $r>0$, we
let $\Delta_{r}$ be the disk centered at the origin in ${\mathbb
C}$ with radius $r$ and $\Delta$ be the disk of unit radius. A map
$$
h(c,z): \Delta \times E\to \C
$$
is called a holomorphic motion of $E$ parametrized by $\Delta$ and
with base point $0$ if
\begin{enumerate}
\item $h(0, z) =z$ for all $z\in E$, \item for every $c\in
\Delta$, $ z \mapsto h(c, z)$ is injective on $\C,$ and \item for
every $z\in E$, $c \mapsto h(c,z)$ is holomorphic for $c$ in
$\Delta$
 \end{enumerate}
We think of $h(c,z)$ as moving through injective mappings with the
parameter $c.$ It starts out at the  identity when $c$ is equal to
the base point  $0$ and moves holomorphically as $c$ varies in
$\Delta$.

We always assume $E$ contains at least three points, $p_1, p_2$
and $p_3.$  Then since the points $h(c,p_1), h(c,p_2)$ and
$h(c,p_3)$ are distinct for each $c \in \Delta,$ there is a unique
M\"obius transformation $B_c$ that carries these three points to
$0, 1,$ and $\infty.$  Since $B_c$ depends holomorphically on $c,$
$\tilde{h}(c,z)=h(c,B_c(z))$ is also a holomorphic motion and it
fixes the points $0, 1, \infty.$  We shall call it a normalized
holomorphic motion.

Holomorphic motions were introduced by M\`a\~n\'e, Sad and
Sullivan in their study of the structural stability problem for
the complex dynamical systems,~\cite{ManeSadSullivan}. They proved
the first result in the topic which is called the $\lambda$-lemma
and which says that any holomorphic motion $h(c,z)$ of $E$
parametrized by $\Delta$ and with base point $0$ can be extended
uniquely to a holomorphic motion of the closure $\overline{E}$ of
$E$ parametrized by $\Delta$ and with the same base point.
Moreover, $h(c,z)$ is continuous on $(c,z)$ and for any fixed $c$,
$z \mapsto h(c,z)$ is quasiconformal on the interior of
$\overline{E}$. Subsequently, holomorphic motions became an
important topic with applications to quasiconformal mapping,
Teichm\"uller theory and complex dynamics. After M\`a\~n\'e, Sad
and Sullivan proved the $\lambda$-lemma,  Sullivan and
Thurston~\cite{SullivanThurston} proved an important extension
result. Namely, they proved that any holomorphic motion of $E$
parametrized by $\Delta$ and with base point $0$ can be extended
to a holomorphic motion of $\C,$ but parametrized by a smaller
disk, namely, by $\Delta_{r}$ for some universal number $0<r<1.$
They showed that $r$ is independent $E$ and independent of the
motion. By a different method and published in the same journal
with the Sullivan-Thurston paper, Bers and
Royden~\cite{BersRoyden} proved that $r \geq 1/3$ for all motions
of all closed sets $E$ parameterized by $\Delta.$ They also showed
that on $\C$ the map $z \mapsto h(c,z)$ is quasiconformal with
dilatation no larger than $(1+|c|)/(1-|c|)$. All of these authors
raised the question as to whether $r=1$ for any holomorphic motion
of any subset of $\C$ parametrized by $\Delta$ and with base point
$0$. In~\cite{Slodkowski} Slodkowski gave a positive answer by
using results from the theory of polynomial hulls in several
complex variables. Other authors~\cite{AstalaMartin}~\cite{Douady}
have suggested alternative proofs.

In this article we give an expository account of a recent proof of
Slodkowski's theorem presented by Chirka in~\cite{Chirka}. (See
also Chirka and Rosay~\cite{ChirkaRosay}.) The method involves an
application of Schauder's fixed point
theorem~\cite{CourantHilbert} to an appropriate operator acting on
holomorphic motions of a point and on showing that this operator
is compact. The compactness depends on the smoothing property of
the Cauchy kernel acting on vector fields tangent to holomorphic
motions.   The main theorem is the following.

\vspace*{10pt}
\begin{theorem}[The Holomorphic Motion Theorem]~\label{hmt}
Suppose $h(c,z): \Delta\times E\to \C$ is a
holomorphic motion of a closed subset $E$ of $\C$ parameterized by
the unit disk. Then there is a holomorphic motion $H (c,z): \Delta
\times \C\to \C$ which extends $h(c,z): \Delta \times E\to\C$.
Moreover, for any fixed $c\in \Delta$, $h(c,\cdot): \C\to\C$ is a
quasiconformal homeomorphism whose quasiconformal dilatation
$$
K(h(c,\cdot)) \leq \frac{1+|c|}{1-|c|}.
$$
The Beltrami coefficient of $h(c, \cdot)$ given by
$$
\mu(c,z)=\frac{\partial h(c, z)}{\partial
\overline{z}}/\frac{\partial h(c, z)}{\partial z}
$$
is a holomorphic function from $\Delta$ into the unit ball of the
Banach space ${\mathcal L}^{\infty}({\mathbb C})$ of all
essentially bounded measurable functions on ${\mathbb C}$.
\end{theorem}

To prove this result we  study the modulus of continuity of
functions in the image of the Cauchy kernel operator.  Then we
apply the Schauder fixed point theorem to a non-linear operator
given by Chirka in~\cite{Chirka}.

After proving this theorem, we show that tangent vectors to
holomorphic motions have  $|\epsilon \log \epsilon|$ moduli of
continuity and then show how this type of continuity for tangent
vectors can be combined with Schwarz's lemma and integration over
the holomorphic variable to produce H\"older continuity on the
mappings.

We also prove that  Kobayashi's and
 Teichm\"uller's metrics on the Teichm\"uller space
$T(R)$ of a Riemann surface coincide.  This method was observed by
Earle, Kra and Krushkal~\cite{EKK}. The result had already been
proved earlier by Royden \cite{Royden} for Riemann surfaces of
finite analytic type and by Gardiner~\cite{Gardiner3} for surfaces
of infinite type.

Finally, we present an application of holomorphic motions to
complex dynamics. In particular, we prove the Fatou linearization
theorem for parabolic germs. A similar type of argument has
recently been used by Jiang in~\cite{Jiang3,Jiang1,Jiang2} and
here we adapt the proof in~\cite[\S3]{Jiang3}. We  believe that
holomorphic motions will provide simplified proofs of many
fundamental results in complex dynamics.

\section{The ${\mathcal P}$-Operator and the Modulus of Continuity}

Let  ${\mathcal C}={\mathcal C}(\CP)$ denote the Banach space of
complex valued, bounded, continuous functions $\phi$  on $\CP$
with the supremum norm
$$
||\phi|| =\sup_{c\in {\CP}} |\phi(c)|.
$$
We use ${\mathcal L}^{\infty}$ to denote the Banach space of
essentially bounded measurable functions $\phi$ on $\CP$ with
${\mathcal L}^{\infty}$-norm
$$
||\phi||_{\infty} =\hbox{ess}\sup_{\CP} |\phi (\zeta)|.
$$

For the theory of quasiconformal mapping we are more concerned
with the action of ${\mathcal P}$ on ${\mathcal L}^{\infty}$. Here
the ${\mathcal P}$-operator is defined by
$$
{\mathcal P} f (c)=-\frac{1}{\pi}\int\int_{\mathbb C}
\frac{f(\zeta)}{\zeta-c}\; d\xi d\eta, \quad \zeta = \xi +i\eta
$$
where $f\in {\mathcal L}^{\infty}$ and has a compact support in
$\CP$. Then
$$
{\mathcal P} f(c)\longrightarrow 0 \quad \hbox{as}\quad
c\longrightarrow\infty.
$$
Furthermore, if $f$ is continuous and has compact support, one can
show that
\begin{equation}\label{deriv}
\frac{\partial ({\mathcal P} f)}{\partial{\overline{c}}} (c)
=f(c), \quad c\in {\mathbb C},
\end{equation}
and by using the notion of generalized
derivative~\cite{AhlforsBers} equation (\ref{deriv}) is still true
Lebesgue almost everywhere if we only know that $f$ has compact
support and is in ${\mathcal L}^p, p \geq 1.$

We first show the classical result that ${\mathcal P}$ transforms
${\mathcal L}^{\infty}$ functions with compact support in $\CP$ to
H\"older continuous functions with H\"older exponent $1-2/p$ for
every $p>2$. See for example~\cite{Ahlforsbook5}. We also show
that ${\mathcal P}$ carries ${\mathcal L}^{\infty}$ functions with
compact supports to functions with an $|\epsilon \log \epsilon|$
modulus of continuity.

\vspace*{10pt}
\begin{lemma}~\label{holder}
Suppose $p>2$ and
$$
\frac{1}{p}+\frac{1}{q} =1,
$$
so that $1<q<2.$  Then for any real number $R>0$, there is a
constant $A_{R}>0$ such that, for any $f\in {\mathcal L}^{\infty}$
with a compact support contained in $\Delta_{R}$,
$$
|| {\mathcal P}f|| \leq A_{R} ||f||_{\infty}
$$
and
$$
|{\mathcal P}f(c)-{\mathcal P}f(c')| \leq A_{R} ||f||_{\infty}
|c-c'|^{1-\frac{2}{p}}, \quad \forall c, c'\in {\mathbb C}.
$$
\end{lemma}

\begin{proof} The norm
$$
||{\mathcal P}f||=\sup_{c\in {\mathbb C}} \frac{1}{\pi} \Big|
\int\int_{{\mathbb C}} \frac{f(\zeta)}{\zeta-c}d\xi d\eta\Big|\leq
\sup_{c\in {\mathbb C}} \frac{1}{\pi} \int\int_{\Delta_{R}}
\frac{|f(\zeta)|}{|\zeta-c|}d\xi d\eta
$$
So
$$
||{\mathcal P}f|| \leq ||f||_{\infty} \sup_{c\in {\mathbb C}}
\frac{1}{\pi} \int\int_{{\Delta_{R}}} \frac{1}{|\zeta-c|} d\xi
d\eta \leq C_{1} ||f||_{\infty}
$$
where
$$
C_{1} = \frac{1}{\pi}  \int\int_{\Delta_{R}} \frac{1}{|\zeta|}d\xi
d\eta  =2R<\infty.
$$

Next
$$
|{\mathcal P}f(c)-{\mathcal P}f(c')| = \frac{1}{\pi} \Big|
\int\int_{{\mathbb C}} f(\zeta) \Big( \frac{1}{\zeta-c}-
\frac{1}{\zeta-c'}\Big) d\xi d\eta\Big|
$$
$$
\leq \frac{|c-c'|}{\pi} \int\int_{\Delta_{R}} \frac{|f(\zeta)|
}{|\zeta-c||\zeta-c'|} d\xi d\eta
$$
$$
\leq \frac{|c-c'|}{\pi} \Big( \int\int_{\Delta_{R}} |f(\zeta)|^{p}
d\xi d\eta\Big)^{\frac{1}{p}} \Big(\int\int_{\Delta_{R}}
\Big|\frac{1}{(\zeta-c)(\zeta-c')}\Big|^{q} d\xi
d\eta|\Big)^{\frac{1}{q}}.
$$
$$
\leq \pi^{\frac{1}{p}-1} R^{\frac{2}{p}} |c-c'| ||f||_{\infty}
\Big( \int\int_{\Delta_{R}}
\Big|\frac{1}{(\zeta-c)(\zeta-c')}\Big|^{q} d\xi
d\eta|\Big)^{\frac{1}{q}} \leq C_{2} ||f||_{\infty}
|c-c'|^{\frac{2}{q}-1}.
$$
where
$$
C_{2} =  \pi^{\frac{1}{p}-1} R^{\frac{2}{p}} \Big(
\int\int_{{\mathbb C}} \Big(\frac{1}{|z||z-1|}\Big)^{q} dx
dy|\Big)^{\frac{1}{q}} <\infty, \quad z=x+iy.
$$
Hence $A_{R}=\max\{ C_{1}, C_{2}\}$ satisfies the requirements of
the lemma.
\end{proof}

Next we prove a stronger form of continuity.

\vspace{10pt}
\begin{lemma}~\label{Zygmund}
Suppose the compact support of $f\in {\mathcal L}^{\infty}$ is
contained in $\Delta$. Then ${\mathcal P}f$ has an $|\epsilon\log
\epsilon|$ modulus of continuity. More precisely, there is a
constant $B$ depending on $R$  such that
$$
|{\mathcal P}f(c)-{\mathcal P}f(c')| \leq ||f||_{\infty} B
|c-c'|\log\frac{1}{|c-c'|}, \quad \forall \; c, c' \in
\Delta_{R},\;\; |c-c'|<\frac{1}{2}.
$$
\end{lemma}

\begin{proof}
Since
$$
|{\mathcal P}f(c)-{\mathcal P}f(c')| = \frac{1}{\pi} \Big|
\int\int_{{\mathbb C}} f(\zeta) \Big( \frac{1}{\zeta-c}-
\frac{1}{\zeta-c'}\Big) d\xi  d\eta\Big|
$$
$$
\leq \frac{1}{\pi}\int\int_{{\mathbb C}} |f(\zeta)| \Big|
\frac{1}{\zeta-c}- \frac{1}{\zeta-c'}\Big| d\xi  d\eta
$$
$$
\leq \frac{|c-c'|\|f\|_{\infty}}{\pi} \int\int_{\Delta} \frac{1
}{|\zeta-c||\zeta-c'|} d\xi  d\eta,
$$
if we put $\zeta'=\zeta-c=\xi '+i\eta'$, then
$$
|{\mathcal P}f(c)-{\mathcal P}f(c')| \leq
\frac{|c-c'|\|f\|_{\infty}}{\pi} \int\int_{\Delta_{1+R}} \frac{1
}{|\zeta'||\zeta'-(c'-c)|} d\xi ' d\eta'.
$$
The substitution $\zeta'' =\zeta' /(c'-c) =\xi ''+i\eta''$ yields
$$ |{\mathcal P}f(c)-{\mathcal P}f(c')| \leq
\frac{|c-c'|\|f\|_{\infty}}{\pi}
\int\int_{\Delta_{\frac{1+R}{|c'-c|}}} \frac{1
}{|\zeta''||\zeta''-1|} d\xi '' d\eta''.
$$
Since $|c-c'|<1/2$, we have $(1+R)/|c'-c|>2$. This implies that
$$
|{\mathcal P}f(c)-{\mathcal P}f(c')|
$$
$$
\leq \frac{|c-c'|\|f\|_{\infty}}{\pi} \left( \int\int_{\Delta_{2}}
\frac{1 }{|\zeta''||\zeta''-1|} d\xi '' d\eta'' +
\int\int_{\Delta_{\frac{1+R}{|c'-c|}}-\Delta_{2}} \frac{1
}{|\zeta''||\zeta''-1|} d\xi'' d\eta'' \right)
$$
Let
$$
C_{3} =\int\int_{\Delta_{2}} \frac{1 }{|\zeta''||\zeta''-1|} d\xi
'' d\eta''.
$$
Then
$$
|{\mathcal P}f(c)-{\mathcal P}f(c')| \leq
\frac{|c-c'|C_{3}\|f\|_{\infty}}{\pi}
+\frac{|c-c'|\|f\|_{\infty}}{\pi}
\int\int_{\Delta_{\frac{1+R}{|c'-c|}}-\Delta_{2}} \frac{1
}{|\zeta''||\zeta''-1|} d\xi '' d\eta''.
$$
If $|\zeta''|>2$ then $|\zeta''-1|>|\zeta''|/2$, and so
$$
\frac{1}{\pi}\int\int_{\Delta_{\frac{1+R}{|c'-c|}}-\Delta_{2}}
\frac{1 }{|\zeta''||\zeta''-1|} d\xi '' d\eta'' \leq \frac{1}{\pi}
\int\int_{\Delta_{\frac{1+R}{|c'-c|}}-\Delta_{2}} \frac{2
}{|\zeta''|^2} d\xi '' d\eta'' $$ $$\leq \frac{1}{\pi}
\int_{0}^{2\pi}\int_{2}^{\frac{1+R}{|c'-c|}} \frac{2 }{r^2} r dr
d\theta = 4\int_{2}^{\frac{1+R}{|c'-c|}} \frac{1 }{r} dr
$$
$$
= 4 \Big(\log\frac{1+R}{|c'-c|}-\log2\Big)= 4(-\log |c-c'| + \log
(1+R) -\log 2).
$$
Thus,
$$
|{\mathcal P}f(c)-{\mathcal P}f(c')| \leq
\frac{|c-c'|C_{3}\|f\|_{\infty}}{\pi} +4|c-c'|\|f\|_{\infty}
(-\log |c-c'| + \log (1+R) -\log 2)
$$
$$
=- |c-c'|\log|c-c'| \Big(\frac{4\pi \log (1+R)+
C_{3}\|f\|_{\infty}-4\pi \log 2}{-\pi \log|c-c'|} +
4\|f\|_{\infty}\Big)
$$
$$
\leq B \Big(- |c-c'|\log|c-c'|)\Big)
$$
where
$$
B= \frac{4\pi \log (1+R)+C_{3}\|f\|_{\infty}- 4\pi \log 2}{\pi
\log2} + 4\|f\|_{\infty}.
$$
\end{proof}

Now we have the following theorem.

\vspace*{10pt}
\begin{theorem}~\label{Zygmundth}
For any $f\in {\mathcal L}^{\infty}$ with a compact support in
$\CP$, ${\mathcal P}f$ has an $|\epsilon\log \epsilon|$ modulus of
continuity. More precisely, for any $R>0$, there is a constant
$C>0$ depending on $R$  such that
$$
|{\mathcal P}f(c)-{\mathcal P}f(c')| \leq C ||f||_{\infty}
|c-c'|\log\frac{1}{|c-c'|}, \quad \forall \; c, c' \in
\Delta_{R},\;\; |c-c'|<\frac{1}{2}.
$$
\end{theorem}

\begin{proof}
Suppose the compact support of $f$ is contained in the disk
$\Delta_{R_{0}}$. Then $g(c) =f(R_{0}c)$ has the compact support
which is contained in the unit disk $\Delta$.

Since
$$
{\mathcal P}g(c) = -\frac{1}{\pi} \int\int_{{\mathbb C}}
\frac{g(\zeta)}{\zeta-c} d\xi  d\eta = -\frac{1}{\pi}
\int\int_{{\mathbb C}} \frac{f(R_{0}\zeta)}{\zeta-c} d\xi  d\eta
=\frac{1}{R_{0}}
 {\mathcal P} f (R_{0}c).
$$
This implies that
$$
{\mathcal P}f(c) = R_{0} {\mathcal P} g\Big( \frac{c}{R_{0}}\Big).
$$
Thus
$$
|{\mathcal P}f(c)-{\mathcal P}f(c')| = R_{0} |{\mathcal P}
g(\frac{c}{R_{0}})-{\mathcal P}g(\frac{c'}{R_{0}})|
$$
$$
\leq R_{0} B ||f||_{\infty} \Big( -\Big|
\frac{c}{R_{0}}-\frac{c'}{R_{0}}\Big| \log
\Big|\frac{c}{R_{0}}-\frac{c'}{R_{0}}\Big|\Big)
$$
$$
= B ||f||_{\infty} \Big(-|c-c'| (\log |c-c'| -\log R_{0})\Big)
$$
$$
= -|c-c'| \log |c-c'| B ||f||_{\infty} \Big( 1- \frac{\log
R_{0}}{\log |c-c'|}\Big)
$$
$$
\leq C ||f||_{\infty} ( -|c-c'| \log |c-c'|)
$$
where
$$
C= B (1+ \frac{\log R_{0}}{\log 2}).
$$
\end{proof}

\section{Extensions of holomorphic motions for $0<r<1$.}

As an application of the modulus of continuity for the ${\mathcal
P}$-operator, we first prove, for any $r$ with $0<r<1,$ that for
any holomorphic motion of a set $E$ parameterized by $\Delta,$
there is an extension to $\Delta_r \times \C.$  We take the idea
of the proof from the recent papers of Chirka~\cite{Chirka} and
Chirka and Rosay,~\cite{ChirkaRosay}.

\vspace*{10pt}

\begin{theorem}~\label{ch} Suppose $E$ is a subset of $\C$
consisting of finite number of points. Suppose $h(c,z):
\Delta\times E\to \C$ is a holomorphic motion. Then for every
$0<r<1$, there is a holomorphic motion $H_{r}(c,z):
\Delta_{r}\times \C\to \C$ which extends $h(c,z): \Delta_{r}\times
E\to\C$.
\end{theorem}

Without loss of generality, suppose
$$
E=\{ z_{0}=0, z_{1}=1, z_{\infty} =\infty, z_{2}, \cdots, z_{n}\}
$$ is a subset of $n+2>3$ points in the Riemann sphere $\C$.
Let $\Delta^{c}$ be the complement of the unit disk in the Riemann
sphere $\C,$ $U$ be a neighborhood of $\Delta^{c}$ in  $\C$ and
suppose
$$
h(c, z): U\times E\to \C
$$
is a holomorphic motion of $E$ parametrized by $U$ and with base
point $\infty.$ Define
$$
f_{i}(c)=h(c,z_i): U \to \C
$$
for $i=0,1,2, \cdots, n, \infty$. We assume the motion is
normalized so
$$
f_{0} (c)=0, \quad f_{1}(c)=1, \quad \hbox{and} \quad
f_{\infty}(c) =\infty, \quad \forall\; c\in U.
$$
Then we have that
\begin{itemize}
\item[a)] $f_{i}(\infty)=z_{i}$,\; $i=2, \cdots, n$;
\item[b)] for any $i=2, \cdots, n$, $f_{i}(c)$ is holomorphic on $U$;
\item[c)] for any fixed $c\in U$, $f_{i}(c)\neq f_{j}(c)$ and
$f_{i}(c)\neq 0,1$, and $\infty$ for $2\leq i\neq j\leq n$.
\end{itemize}

Since $\Delta^{c}$ is compact, $f_{i}(c)$ is a bounded function on
$\Delta^{c}$ for every $2\leq i\leq n$ and so there is a constant
$C_{4}>0$ such that
$$
|f_{i}(c)|\leq C_{4}, {\rm \ for \ all \ }  c \in \Delta^{c} {\rm
\ and \ all \ } i {\rm \ with \ }  2\leq i\leq n.
$$
Moreover, there is a number $\delta >0$ such that
$$
\mid f_{i}(c)-f_{j}(c) \mid > \delta, {\rm \ for \ all \ } i {\rm
\ and \ } j {\rm \ with \ } 2\leq i\neq j\leq n, {\rm \ and \ for
\ all \ } c \in \Delta^{c}.
$$
We extend the functions $f_{i}(c)$ on $\Delta^{c}$ to continuous
functions on the Riemann sphere $\C$ by defining
$$
f_{i}(c) =f_{i} \Big( \frac{1}{\overline{c}} \Big), {\rm \ for \
all \ } c \in  \overline{\Delta}.
$$
We still have
$$
\mid f_{i}(c)-f_{j}(c) \mid > \delta, {\rm \ for \ all \ } i {\rm
\ and \ } j {\rm \ with \ } 2\leq i\neq j\leq n {\rm \ and \ for \
all \ } c \in \C
$$ and
$$
|f_i(c)| \leq C_4 {\rm \ for \ all \ } i {\rm \ and \ } j {\rm \
with \ } 2\leq i\neq j\leq n {\rm \ and \ for \ all \ } c \in \C.
$$
Since $f_{i}(c)$ is holomorphic in $\Delta^{c}$ and
$f_{i}(\infty)=z_{i}$, the series expansion of $f_{i}(c)$ at
$\infty$ is
$$
f_{i}(c)=z_{i}+\frac{a_{1}}{c} +\frac{a_{2}}{c^{2}}+ \cdots
+\frac{a_{n}}{c^{n}}+\cdots, \quad c \in \Delta^{c}, \quad
\forall\; c\in \Delta^{c}.
$$
This implies that
$$
f_{i}(c)=f_{i}\Big( \frac{1}{\overline{c}}\Big)
=z_{i}+a_{1}\overline{c} +a_{2}(\overline{c})^{2}+\cdots a_{n}
(\overline{c})^{n}+\cdots , \quad \forall \; c\in
\overline{\Delta}.
$$
We have that
$$
\frac{\partial f_{i}}{\partial \overline{c}} (c) = a_{1} + 2 a_{2}
\overline{c}+ \cdots +n a_{n} (\overline{c})^{n-1}+\cdots
$$
exists at $c=0$ and  is a continuous function on
$\overline{\Delta}$. Furthermore, $(\partial f_{i}/\partial
\overline{c}) (c)=0$ for $c\in \Delta^{c}$. Since
$\overline{\Delta}$ is compact, there is a constant $C_{5}>0$ such
that
$$
\mid \frac{\partial f_{i}}{\partial \overline{c}}(c) \mid \leq
C_{5}, \quad \forall\; c\in \C, \quad \forall\; 2\leq i\leq n.
$$
%Note that $(\partial f_{i}/\partial \overline{c})(c)$ may have a
%first-type discontinuity point on the boundary $\partial \Delta$.

Pick a $C^{\infty}$ function $0\leq \lambda(x)\leq 1$ on ${\mathbb
R}^{+}=\{ x\geq 0\}$ such that $\lambda(0)=1$ and $\lambda(x)=0$
for $x\geq \delta/2$. Define
\begin{equation}\label{magicsum}
\Phi(c,w)=\sum^{n}_{i=2} \lambda(\mid
w-f_{i}(c)\mid)\frac{\partial f_{i}}{\partial \overline{c}}(c),
\quad (c,w)\in \C\times {\mathbb C}.
\end{equation}

\vspace*{10pt}
\begin{lemma}\label{items}
The function $\Phi(c,w)$ has the following properties:
\begin{itemize}
\item[i)] only one term in the sum (\ref{magicsum}) defining
$\Phi(c,w)$ can be nonzero,
%exists only one $2\leq i\leq n$ such that $$
%\Phi(c,w)=\lambda(\mid w-f_{i}(c)\mid)\frac{\partial
%f_{i}}{\partial \overline{c}}(c); $$
\item[ii)] $\Phi(c,w)$ is uniformly bounded by $C_{5}$ on
$\C\times {\mathbb C}$, \item[iii)] $\Phi(c,w)=0$ for $(c, w)\in
\Big((\overline{\Delta})^{c} \times {\mathbb C}\Big) \cup \Big(
\C\times (\overline{\Delta}_{R})^{c}\Big)$ where
$R=C_{4}+\delta/2$, \item[iv)] $\Phi(c,w)$ is a Lipschitz function
in $w$-variable with a Lipschitz constant $L$ independent of
$c\in\hat{\mathbb C}$.
\end{itemize}
\end{lemma}

\begin{proof}
Item i) follows because if a point $w$ is within distance
$\delta/2$ of one of the values $f_{i}(c)$ it must be at distance
greater than $\delta/2$ from any of the other values $f_{j}(c).$
Item ii) follows from item i) because there can be only one term
in (\ref{magicsum}) which is nonzero and that term is bounded by
the bound on $\frac{\partial f_j(c)}{\partial \overline{c}}.$ Item
iii) follows because if $c\in (\overline{\Delta})^{c}$, then
$(\partial f_{i}/\partial \overline{c})(c)=0$, and if $w\in
(\overline{\Delta}_R)^{c}$, $\Phi(c,w)=0$. To prove  item iv), we
note that there is a constant $C_{6}>0$ such that $|\lambda
(x)-\lambda (x')|\leq C_{6}|x-x'|$. Since $|(\partial
f_{i}/\partial \overline{c})(c)|\leq C_{5}$,
\begin{equation}\label{Lip}
|\Phi (c, w) -\Phi(c,w')| \leq C_{6} C_{5} \sum_{i=2}^{n} \Big|\
|w-f_{i}(c)|-|w'-f_{i}(c)|\Big|.
\end{equation}
Since only one of the terms in the sum (\ref{magicsum}) for
$\Phi(c,w)$ is nonzero and possibly a different term is non-zero
in the sum for $\Phi(c,w'),$ we obtain
$$
 |\Phi (c, w) -\Phi(c,w')| \leq  2 C_{6} C_{5} |w-w'|.
$$
Thus $L=2 C_{5}C_{6}$ is a Lipschitz constant independent of $c\in
\hat{\mathbb C}$.
\end{proof}

Since $\Phi(c,f(c))$ is an ${\mathcal L}^{\infty}$ function with a
compact support in $\overline{\Delta}$ for any $f\in {\mathcal
C}$, we can define an operator ${\mathcal Q}$ mapping functions in
${\mathcal C}$ to functions in ${\mathcal L}^{\infty}$ with
compact support by
$$
{\mathcal Q}f (c)= \Phi(c,f(c)), \quad f(c) \in {\mathcal C}.
$$
Since $\Phi (c,w)$ is Lipschitz in the $w$ variable with a
Lipschitz constant $L$ independent of $c\in\C$, we have $$
|{\mathcal Q}f (c) - {\mathcal Q}g (c) | =  |\Phi(c,f(c))-\Phi (c,
g(c))| \leq L |f(c)-g(c)|.
$$
Thus
$$
||{\mathcal Q}f -{\mathcal Q}g||_{\infty} \leq L ||f-g||
$$
and ${\mathcal Q}: {\mathcal C} \to {\mathcal L}^{\infty}$ is a
continuous operator.

From Lemma~\ref{holder},
$$
||{\mathcal P}f||\leq A_{1} ||f||_{\infty}
$$
for any $f\in {\mathcal L}^{\infty}$ whose compact support is
contained in $\Delta$, and so the composition ${\mathcal
K}={\mathcal P}\circ {\mathcal Q}$, where
$$
{\mathcal K}f(c)=-\frac{1}{\pi}\int\int_{\mathbb
C}\frac{\Phi(\zeta,f(\zeta))}{\zeta-c}d\xi  d\eta, \quad \zeta
=\xi +i\eta,
$$
is  a continuous operator from  ${\mathcal C}$ into itself.

\vspace*{10pt}
\begin{lemma}~\label{koperator1}
There is a constant $D>0$ such that
$$
||{\mathcal K}f || \leq D, \quad \forall\; f\in {\mathcal C};
$$
\end{lemma}

\begin{proof}
Since $\Phi (c, w)=0$ for $c\in \Delta^{c}$ and since $\Phi (c,
w)$ is bounded by $C_{5}$, we have that
$$
|{\mathcal K}f(c)|= \Big|
\frac{1}{\pi}\int\int_{\C}\frac{\Phi(\zeta,f(\zeta))}{\zeta-c} \;
d\xi  d\eta\Big|= \Big| \frac{1}{\pi}\int\int_{\Delta}
\frac{\Phi(\zeta,f(\zeta))}{\zeta-c} \; d\xi d\eta \Big|
$$
$$
\leq  \frac{1}{\pi}\int\int_{\Delta}
\frac{|\Phi(\zeta,f(\zeta))|}{|\zeta-c|} \; d\xi d\eta
$$
$$
\leq \frac{C_{5}}{\pi}\int\int_{\Delta} \frac{1}{|\zeta-c|} \;
d\xi d\eta \leq 2C_{5}=D
$$
where $\zeta=\xi +i\eta$.
\end{proof}

\vspace{10pt}
\begin{lemma}~\label{holderk}
Suppose $p>2$ and $q$ is the dual number between $1$ and $2$
satisfying
$$
\frac{1}{p}+\frac{1}{q} =1.
$$
Then for any $f\in {\mathcal C}$, ${\mathcal K}f$ is
$\alpha$-H\"older continuous for
$$
0<\alpha =\frac{2}{q} -1<1
$$
with a H\"older constant $H= A_{1} C_{5}$ independent of $f$.
\end{lemma}

\begin{proof} From Lemma~\ref{holder},
$$
|{\mathcal K}f(c) -{\mathcal K}f(c')| = |{\mathcal P} ({\mathcal
Q}f)(c) -{\mathcal P} ({\mathcal Q}f)(c')|
$$
$$
\leq A_{1} ||{\mathcal Q}f||_{\infty} |c- c'|^{\alpha} \leq
A_{1}C_{5} |c- c'|^{\alpha} =H |c- c'|^{\alpha}.
$$
\end{proof}

The above two lemmas imply that ${\mathcal K}: {\mathcal C}\to
{\mathcal C}$ is a continuous compact operator. Now for any $z\in
{\mathbb C}$, let
$$
{\mathcal B}_{z} = \{ f\in {\mathcal C} \;|\; ||f|| \leq |z|+D\}.
$$
It is a bounded convex subset in ${\mathcal C}$. The continuous
compact operator $z+{\mathcal K}$ maps ${\mathcal B}_{z}$ into
itself. From the Schauder fixed point
theorem~\cite{CourantHilbert}, $z+{\mathcal K}$ has a fixed point
in ${\mathcal B}_{z}$. That is, there is a $f_{z} \in {\mathcal
B}_{z}$ such that
$$
f_{z}(c) =z+ {\mathcal K}f_{z}(c), \quad \forall\; c\in {\mathbb
C}.
$$
Since ${\mathcal Q}f(c)$ has a compact support in
$\overline{\Delta}$ for any $f\in {\mathcal C}$, ${\mathcal
K}f_{z}(c) \to 0$ as $c\to \infty$. So $f_{z}$ can be extended
continuously to $\infty$ such that $f_{z}(\infty)=z$.

\vspace*{10pt}
\begin{lemma}~\label{uniqueness}
The solution $f_{z}(c)$ is the unique fixed point of the operator
$z+{\mathcal K}$.
\end{lemma}

\begin{proof}
Suppose $f_{z}(c)$ and $g_{z}(c)$ are two solutions. Take
$$
\phi(c)=f_{z}(c)-g_{z}(c) ={\mathcal K}(f_{z}) (c)- {\mathcal K}
(g_{z})(c).
$$
Then $\phi(c)\to 0$ as $c\to
\infty$. Now
$$
\frac{\partial \phi}{\partial \overline{c}}(c)=\frac{\partial
f_{z}}{\partial \overline{c}}(c)-\frac{\partial g_{z}}{\partial
\overline{c}}(c)=\Phi(c,f_{z}(c))-\Phi(c,g_{z}(c)).
$$
So by Lemma \ref{items}
$$
\frac{\partial \phi}{\partial \overline{c}}(c)=0, \quad \forall\;
c\in \Delta^{c}.
$$

Since $\Phi(c,w)$ is Lipschitz in $w$-variable with a Lipschitz
constant $L$,
$$
| \frac{\partial \phi}{\partial \overline{c}}(c)|=|
\Phi(c,f_{z}(c))-\Phi(c,g_{z}(c))| \leq L |f_{z}(c)-g_{z}(c)|=L
|\phi(c)|.
$$
Assuming that $\phi(c)$ is not equal to zero,  define
$$
\psi (c)=- \frac{\frac{\partial \phi}{\partial
\overline{c}}(c)}{\phi(c)},
$$
and otherwise, define $\psi(c)$ to be equal to zero. Then $\psi
(c)$ is a function in ${\mathcal L}^{\infty}$ with a compact
support in $\overline{\Delta}$. So we have ${\mathcal P}\psi$ in
${\mathcal C}$ such that
$$
\frac{\partial {\mathcal P} \psi}{\partial \overline{c}} (c) =
\psi(c).
$$

Consider $e^{{\mathcal P}\psi} \cdot \phi$. Then
$$
\frac{ \partial (e^{{\mathcal P} \psi }\cdot \phi)}{\partial
\overline{c}}\equiv 0.
$$
This means that $e^{{\mathcal P}\psi }\cdot \phi $ is holomorphic
on the complex plane ${\mathbb C}$.

When $c\longrightarrow \infty$, ${\mathcal P}\psi\longrightarrow
0$ and $\phi(c)\longrightarrow 0$. This implies that $e^{{\mathcal
P} \psi}\cdot \phi$ is bounded on ${\mathbb C}$. So $e^{{\mathcal
P}\psi}\cdot \phi$ is a constant function. But $\phi(\infty)=0$,
so $e^{{\mathcal P}\psi}\cdot \phi\equiv 0$. Thus $\phi(c)\equiv
0$ and $f_{z}(c)=g_{z}(c)$ for all $c\in {\mathbb C}$.
\end{proof}

For $z_{i}\in E$, $2\leq i\leq n$, consider
$$
{\mathcal K} f_{i}(c)=-\frac{1}{\pi} \int\int_{{\mathbb
C}}\frac{\Phi(\zeta,f_{i}(\zeta))}{\zeta-c}d\xi  d\eta,
$$
where $\zeta =\xi +i\eta$. From the definition of $\Phi(c,w)$, we
have that
$$
\Phi(\zeta,f_{i}(\zeta))=\frac{\partial f_{i}}{\partial
\overline{\zeta}} (\zeta).
$$
So
$$
{\mathcal K}f_{i}(c)=-\frac{1}{\pi}\int\int_{\mathbb
C}\frac{\frac{\partial f_{i}}{\partial
\overline{\zeta}}(\zeta)}{\zeta-c}d\xi  d\eta.
$$
This implies that
$$
\frac{\partial {\mathcal
K}f_{i}}{\partial\overline{c}}(c)=\frac{\partial f_{i}}{\partial
\overline{c}}(c)
$$
and that
$$
\frac{\partial (f_{i}- {\mathcal K} f_{i})}{\partial\overline{c}}
(c)\equiv 0.
$$
So $f_{i}(c)-{\mathcal K}f_{i}(c)$ is holomorphic on ${\mathbb
C}$. When $c\longrightarrow \infty$, $f_{i}(c)\longrightarrow
z_{i}$ and ${\mathcal K}f_{i}(c)\longrightarrow 0$. So
$f_{i}(c)-{\mathcal K}f_{i}(c)$ is bounded. Therefore it is a
constant function. We get that
$$
f_{i}(c)=z_{i}+{\mathcal K}f_{i}(c).
$$
Thus from Lemma~\ref{uniqueness}, $f_{i}(c)=f_{z_{i}}(c)$ for all
$c\in \C$.

By defining $H(c, z)=f_{z}(c)$ for $(c, z) \in
\overline{\Delta}^{c}\times {\mathbb C}\setminus \{ 0, 1\}$ and
$H(c,0)=0$ and $H(c, 1)=1$ and $H(c,\infty)=\infty$, we get a map
$$
H(c, z)=f_{z}(c): \overline{\Delta}^{c}\times \C\to \C,
$$
which is an extension of
$$
h(c,z):\overline{\Delta}^{c}\times E\to \C.
$$

\vspace*{10pt}
\begin{lemma}~\label{holmot}
The map
$$ H(c, z)=f_{z}(c): \overline{\Delta}^{c}\times
\C\to \C,
$$
is a holomorphic motion.
\end{lemma}

\begin{proof}
First $H(\infty, z) =f_{z}(\infty)=z$ for all $z\in \C$. From the
fixed point equation
$$
H(c, z) =z +{\mathcal K} H(c, z),
$$
$$
\frac{\partial H(c, z)}{\partial \overline{c}} = \Phi (c, H(c,z)).
$$
Since $\Phi (c, w) =0$ for all $c\in \overline{\Delta}^{c}$,
$$
\frac{\partial H(c, z)}{\partial \overline{c}}=0, \quad \forall\;
c\in \overline{\Delta}^{c}.
$$
Thus, for any fixed $z\in \C$, $H(c,z): \overline{\Delta}^{c}\to
\C$ is holomorphic.

For any two $z \neq z'\in \C$, we claim that $H(c, z)\neq H(c,
z')$ for all $c\in {\mathbb C}$. This  implies that for any fixed
$c\in \overline{\Delta}^{c}$, $H(c,z)$ is an injective map on $z
\in \C$ and that $H(c,z)$ is a holomorphic motion. To prove the
claim take any two $z,z'\in \C$. Assume there is a point $c_{0}
\in \C$ such that $H(c_{0}, z)= H(c_{0}, z')$. If $c_{0}=\infty$,
then $z=z',$ because by assumption the holomorphic motion starts
out at the identity. If $c_{0}\neq \infty,$ then
$$
f_{z}(c_{0})- f_{z'}(c_{0})=(z-z')+{\mathcal K}
f_{z}(c_{0})-{\mathcal K}f_{z'}(c_{0}),
$$
and we can repeat the same argument we have given in Lemma
\ref{uniqueness}.

 Let
$\phi(c)=f_{z}(c)-f_{z'}(c)$. Then $\phi(c_{0})=0$. However,
$$
\frac{\partial \phi}{\partial \overline{c}}(c)=\frac{\partial
f_{z}}{\partial \overline{c}}(c)-\frac{\partial f_{z'}}{\partial
\overline{c}}(c)=\Phi(c,f_{z}(c))-\Phi(c,f_{z'}(c)).
$$
This implies that
$$
\frac{\partial \phi}{\partial \overline{c}}(c)=0
$$
for $c\in \overline{\Delta}^{c}$. Since $\Phi(c,w)$ is Lipschitz
in $w$-variable with a Lipschitz constant $L$,
$$
|\frac{\partial \phi}{\partial \overline{c}}
(c)|=|\Phi(x,f_{z}(c))-\Phi(c,f_{z'}(c))| \leq L |
f_{z}(c)-f_{z'}(c)|=L |\phi(c)|.
$$
If $\phi(c)\neq 0$, define
$$
\psi(c)=-\frac{\frac{\partial
\phi}{\partial \overline{c}}(c)}{\phi(c)},
$$
otherwise, define $\psi(c)=0$. Then
$$
\frac{ \partial e^{{\mathcal P}\psi}\cdot \phi}{\partial
\overline{c}}(c)\equiv 0.
$$
So $e^{{\mathcal P}\psi}\cdot \phi$ is holomorphic on ${\mathbb
C}$. When $c\longrightarrow \infty$, ${\mathcal P}
\psi(c)\longrightarrow 0$ and $\phi(c)\longrightarrow z-z'$. So
$e^{{\mathcal P\psi(c)}}\cdot \phi(c)$ is bounded on ${\mathbb
C}$. This implies that $e^{{\mathcal P}\psi(c)}\cdot \phi(c)$ is a
constant function. Since $\phi(c_{0})=0$, $e^{{\mathcal
P}\psi(c)}\cdot \phi(c)\equiv 0$. So $z=z'$.
\end{proof}

\begin{proof}[Proof of Theorem~\ref{ch}]
Suppose
$$
h(c,z): \Delta\times E\to \C
$$
is a holomorphic motion. For every $0<r<1$, consider
$\alpha_{r}(c) = r/c$. Let $U_{r} = \alpha_{r} (\Delta_{r})
\supset \overline{\Delta}^{c}$. Then
$$
h_{r} (\alpha_{r}^{-1}(c), z): U_{r}\times E\to \C
$$
is a holomorphic motion. From Lemmas~\ref{uniqueness}
and~\ref{holmot}, it can be extended to a holomorphic motion
$$
\tilde{H}_{r}(c, z): \overline{\Delta}^{c}\times \C\to \C.
$$
Then
$$
H_{r}(c,z)= \tilde{H}(\alpha_{r}(c), z): \Delta_{r}\times \C\to \C
$$
is a holomorphic motion which is an extension of $h(c,z)$ on
$\Delta_{r}\times E$.
\end{proof}

\section{Controlling quasiconformal dilatation}

To control the quasiconformal dilatation of a holomorphic motion
there are two methods available. One is given by the Bers-Royden
paper~\cite{BersRoyden} and the other is obtained by combining
methods given in the Bers-Royden paper and in the
Sullivan-Thurston paper~\cite{SullivanThurston}. We discuss the
latter method first.

Consider a set of four points $S=\{z_{1}, z_{2}, z_{3}, z_{4}\}$
in $\C$. These points are distinct if an only if the cross ratio
$$
Cr(S) =
\frac{z_{1}-z_{3}}{z_{1}-z_{4}}:\frac{z_{2}-z_{3}}{z_{2}-z_{4}}=
\frac{z_{1}-z_{3}}{z_{1}-z_{4}}\frac{z_{2}-z_{4}}{z_{2}-z_{3}}
$$
is not equal to  $0, 1,$ or $\infty$. If one of these points is
equal to $\infty$, say $z_{4}$, then this cross ratio becomes a
ratio
$$
Cr(S) = \frac{z_{1}-z_{3}}{z_{2}-z_{3}}.
$$

Suppose $H:\C\mapsto\C$ is an orientation-preserving homeomorphism
such that $H(\infty)=\infty$. Then one of the definitions of
quasiconformality~\cite{LehtoVirtanenbook} of $H$ is that
 $$ \lim_{r\to 0} \sup_{a\in
{\mathbb C}} \frac{\sup_{|z-a|=r}
|H(z)-H(a)|}{\inf_{|z-a|=r}|H(z)-H(a)|} <\infty.
$$
In~\cite{SullivanThurston} Sullivan and Thurston used this
definition to prove the following theorem.

\vspace{10pt}

\begin{theorem}~\label{qc}
Suppose $H(c, z): \Delta\times \C\to \C $ is a normalized
holomorphic motion of $\C$ parametrized by $\Delta$ and with base
point $0$. Then for each $c_{0}\in \Delta$, the map
$h(c_{0},\cdot): \C\mapsto \C$ is quasiconformal.
\end{theorem}

\begin{proof}
Let $a\in {\mathbb C}$ be any point. Let $z_{3}=a$. Let $z_{1}$
and $z_{2}$ be two distinct points in ${\mathbb C}$ not equal to
$a$ and $z_{4}=\infty$. Then the cross ratio $Cr(S)
=(z_{1}-z_{3})/(z_{2}-z_{3})$.

Now consider $z_{1}(c) =H(c, z_{1})$, $z_{2}(c) =H(c, z_{2})$,
$z_{3}(c) =H(c, z_{3})$, and $z_{4}(c) =H(c, z_{4})=\infty$ and
$S(c) =\{z_{1}(c), z_{2}(c), z_{3}(c)\}$. The cross ratio
$$
Cr(S(c)) = \frac{z_{1}(c)-z_{3}(c)}{z_{2}(c)-z_{3}(c)}.
$$
Since $H(c,z)$ is a holomorphic motion, $Cr(S(c)): \Delta\mapsto
{\mathbb C}\setminus \{0,1\}$ is a holomorphic function. Then it
decreases the hyperbolic distances from $\rho_{\Delta}$ to
$\rho_{0,1}$. So
$$
\rho_{0,1} (Cr(S(c_{0})), Cr(S)) \leq \rho_{\Delta}(0, c_{0}) =
\log \frac{1+|c_{0}|}{1-|c_{0}|}.
$$
This implies that there is a constant $K=K(c_{0})>0$ such that for
any $|Cr(S)|=1$,
$$
|Cr(S(c_{0})| \leq K.
$$
So we have that
$$
\lim_{r\to 0} \sup_{a\in {\mathbb C}} \frac{\sup_{|z-a|=r}
|H(c_{0}, z)-H(c_{0}, a)|}{\inf_{|z-a|=r}|H(c_{0}, z)-H(c_{0},
a)|} <\infty,
$$
that is, $H(c_{0},z)$ is quasiconformal.
\end{proof}

Suppose ${\mathcal L}^{\infty}(W)$ is the Banach space of all
essentially bounded measurable functions on $W$ equipped with
$\|\cdot\|_{\infty}$-norm.  Bers and Royden~\cite{BersRoyden}
proved the following theorem.

\vspace*{10pt}
\begin{theorem} \label{hol}
Suppose $h(c, z): \Delta\times E\to \hat{\mathbb C} $ is a
holomorphic motion of $E$ parametrized by $\Delta$ and with base
point $0$ and $E$ has nonempty interior $W$, then the Beltrami
coefficient of $h(c, \cdot)|_{W}$ given by
$$
\mu(c,z)=\frac{\partial h(c, z)|_{W}}{\partial
\overline{z}}/\frac{\partial h(c, z)|_{W}}{\partial z}
$$
is a holomorphic function mapping $c\in\Delta$ into the unit ball
of the Banach space ${\mathcal L}^{\infty}(W)$.
\end{theorem}

\begin{proof}
Since the dual of the Banach space ${\mathcal L}^{1}(W)$ of
integrable functions on $W$ is ${\mathcal L}^{\infty}(W),$ to
prove $\mu (c, \cdot)$ is a holomorphic map, it suffices to show
that the function
$$
c \mapsto \Psi(c)=\int\int_{W}\alpha(z)\mu_{c}(z)dxdy
$$
is holomorphic in $\Delta$ for every $\alpha(z)\in {\mathcal
L}^{1}(W)$. Furthermore, it suffices to check this for every
$\alpha(z)\in {\mathcal L}^{1}(W)$ with a compact support in $W$.

Suppose $\alpha(z)\in {\mathcal L}^{1}(W)$ has a compact support
$\hbox{supp}(\alpha)$ in $W$. There is an $\epsilon >0$ such that
the $\epsilon$-neighborhood
$U_{\epsilon}(\hbox{supp}(\alpha))\subset W$. From
Theorem~\ref{qc}, $h(c,\cdot)$ is quasiconformal, it is
differentiable, a.e. in $W$. Thus
$$
\Psi(c)=\int\int_{\hbox{supp}(\alpha)}\alpha(z)\frac{h_{x}(c,z)+ih_{y}(c,z)}{h_{x}(c,z)-ih_{y}(c,z)}dxdy
$$
$$
\Psi(c)=\int\int_{\hbox{supp}(\alpha)}\alpha(z)\frac{1+i\frac{h_{y}(c,z)}{h_{x}(c,z)}}{1-i\frac{h_{y}(c,z)}{h_{x}(c,z)}}dxdy
$$
$$
\Psi(c)=\int\int_{\hbox{supp}(\alpha)}\alpha(z)\lim_{\lambda \to
0}\frac{1+i\sigma_{c}(z,\lambda)}{1-i\sigma_{c}(z,\lambda)}dxdy
$$
where
$$\sigma_{c}(z,\lambda)=\frac{h(c,z+i\lambda)-h(c,z)}{h(c,z+\lambda)-h(c,z)}.$$

For any fixed $z \neq 0,1,\infty$ and $\lambda$ small,
$$
\varrho (c)=\sigma_{c} (z): \Delta \mapsto \C\setminus
\{0,1,\infty\}
$$
is a holomorphic function of $c\in \Delta$. So it decreases the
hyperbolic distances on $\Delta$ and on $\C\setminus
\{0,1,\infty\}$. Since $\varrho(0)=i$, there is a number $0<r<1$
such that for
$$
|\sigma_{c}(z,\lambda)-i|\leq \frac{1}{2}, \quad |c|<r.
$$
Therefore
$$
\Big|\frac{1+i\sigma_{c}(z,\lambda)}{1-i\sigma_{c}(z,\lambda)}\Big|
=\Big|
\frac{-i+\sigma_{c}(z,\lambda)}{i+\sigma_{c}(z,\lambda)}\Big|\leq
\frac{\frac{1}{2}}{\frac{3}{2}}=\frac{1}{3}
$$

By the Dominated Convergence Theorem, for $|c|<r$, the sequence of
holomorphic functions
$$
\Psi_{n}(c)=\int\int_{\hbox{supp}(\alpha)}\alpha(z)
\frac{1+i\sigma_{c}(z,\frac{1}{n})}{1-i\sigma_{c}(z,\frac{1}{n})}dxdy
$$
converges uniformly to $\Psi(c)$ as $n\to \infty$. Thus $\Psi(c)$
is  holomorphic for $|c|<r$ and this implies that
$$
\mu(c, \cdot): \{c\; |\; |c|<r\} \to  {\mathcal L}^{\infty}(W)
$$
is holomorphic.

Now consider arbitrary $c_{0}\in \Delta$. Let $s=1-|c_{0}|$ and
let
$$
E_{0}=h(c_{0},E)\quad \hbox{and}\quad W_{0}=h(c_{0},W)
$$
and
$$
g(\tau,\zeta)=h(c_{0}+s\tau, z), \quad \zeta=h(c_{0},z).
$$
Then $W_{0}$ is the interior of $E_{0}$ since $h(c,z)$ is a
quasiconformal homeomorphism. Also
$$
g:\Delta\times E_{0}\to \C
$$
is a holomorphic motion. So the Beltrami coefficient of $g$ is a
holomorphic function on $\{ \tau \;| \; |\tau|<r\}$. Hence the
Beltrami coefficient of $h$ is a holomorphic function on $\{ c |
|c-c_{0}|<sr\}$. This concludes the proof.
\end{proof}

\vspace{10pt}
\begin{theorem}~\label{control}
Suppose $h(c, z): \Delta\times E\to \C $ is a holomorphic motion
of $E$ parametrized by $\Delta$ and with base point $0$ and
suppose $E$ has a nonempty interior $W$. Then for each
$c\in\Delta$, the map $h(c, z)|_{W}$ is a $K$-quasiconformal
homeomorphism of $W$ into $\C$ with
$$
K \leq \frac{1+|c|}{1-|c|}.
$$
\end{theorem}

\begin{proof}
Since $\mu (c, \cdot) : \Delta\mapsto {\mathcal L}^{\infty} (W)$
is a holomorphic map and since $\mu (0,\cdot)=0$. From the
Schwarz's lemma, $\|\mu\|_{\infty}\leq |c|$. This implies that the
quasiconformal dilatation of $h(c, \cdot)$ is less than or
equation to $K=\frac{1+|c|}{1-|c|}$.
\end{proof}

\section{Extension of holomorphic motions for $r=1$.}

\vspace*{10pt}
\begin{theorem}[Slodkowski's Theorem]~\label{fullextension}
Suppose $h(c,z): \Delta\times E\to \C$ is a holomorphic motion.
Then there is a holomorphic motion $H (c,z): \Delta \times \C\to
\C$ which extends $h(c,z): \Delta \times E\to\C$.
\end{theorem}

\begin{proof}
Suppose $E$ is a subset of $\C$. Suppose $h(c,z): \Delta\times
E\to \C$ is a holomorphic motion. Let $E_{1}$, $E_{2}...$ be a
sequence of nested subsets consisting of finite number of points
in $E$. Suppose
$$
\{0,1,\infty\}\subset E_{1}\subset E_{2}\subset \cdots \subset E
$$
and suppose $\cup_{i=1}^{\infty} E_{i}$ is dense in $E$. Then
$h(c,z): \Delta\times E_{i}\to \C$ is a holomorphic motion for
every $i=1, 2, \ldots$.

From Theorem~\ref{ch}, for any $0<r<1$ and $i\geq 1$, there is a
holomorphic motion $H_{i}(c,z): \Delta_{r}\times \C\mapsto \C$
such that $H_{i} |\Delta_{r}\times E_{i} = h|\Delta_{r}\times
E_{i}.$  From Theorem~\ref{control}, $z \mapsto H_{i}(c, z)$ is
$(1+|c|/r)/(1-|c|/r)$-quasiconformal and fixes $0,1,\infty$ for
all $i>0$. So for any $|c| \leq r$, the functions $z \mapsto
H_{i}(c, z)$ form a normal family and there is a subsequence
$H_{i_{k}}(c,\cdot)$ converging uniformly (in the spherical
metric) to a $(1+|c|/r)/(1-|c|/r)$-quasiconformal homeomorphism
$H_{r}(c,\cdot):\C \to \C$ such that $H_{r} (c,z)=h(c,z)$ for
$z\in \cup (E_{j_{k}})$.

Let $\zeta$ be a point in $E$. Replacing $E_{i}$ by $E_{i}\cup
\{\zeta\}$ and repeating the previous construction we obtain a
$(1+|c|/r)/(1-|c|/r)$-quasiconformal homeomorphism $\tilde{H}_{r}$
which coincides with $h(c,z)$ on $\cup E_{i_{k}}\cup \{\zeta\}.$
But $z \mapsto H_{r}(c,z)$ and $z \mapsto \tilde{H}_{r}(c,z)$ are
continuous everywhere and coincide on $\cup E_{i_{k}}$, hence on
$E$. So $H_{r}(c,\zeta)=\tilde{H}_{r}(c,\zeta)=h(c,\zeta)$ for any
$\zeta \in E$.

Now for any $z\neq 0,1,\infty$, since $H_{i}(c, z): \Delta\mapsto
\C$ are holomorphic and omit three points $0,1,\infty$. So the
functions $c \mapsto H_i(c,z)$ form a normal family.  Any
convergent subsequence $H_{i_{k}}(c,z)$ still has a holomorphic
limit $H_{r}(c,z)$, thus $H_{r}(c,z): \Delta_{r}\times \C\mapsto
\C$ is a holomorphic motion which extends $h(c,z)$ on
$\Delta_{r}\times \C$.

Now we are ready to take the limit as $r \rightarrow 1$. For each
$0<r<1$, let $H_{r}(c, z): \Delta_{r}\times \C\to \C$ be a
holomorphic motion such that $H_{r}=h$ on $\Delta_{r}\times E$.
From Theorem~\ref{control}, $H_{r}(c, \cdot)$ is
$(1+|c|/r)/(1-|c|/r)$-quasiconformal for every $c$ with $|c|\leq
r$.

Take a sequence $Z=\{ z_{i}\}_{i=1}^{\infty}$ of points in $\C$
such that $\overline{Z}=\C$, and assume $0$, $1$, and $\infty$ are
not elements of $Z$. For each $i=1, 2, \cdots$, $H_{r}(c, z_{i}):
\Delta_{r}\to \C$ is holomorphic and omits $0,1, \infty$. Thus
$\{H_{r}(c, z_{i}), c\in \Delta_{r}\}_{0<r<1}$ forms a normal
family. We have a subsequence $r_{n}\to 1$ such that $H_{r_{n}}
(c, z_{i})$ tends to a holomorphic function $\tilde{H}(c, z_{i})$
defined on $\Delta$ uniformly on the spherical metric for all
$i=1,2, \cdots$. For a fixed $c\in \Delta$, $H_{r_{n}} (c, \cdot)$
are $(1+|c|/r_{n})/(1-|c|/r_{n})$-quasiconformal for all
$r_{n}>|c|$. So $\{ H_{r_{n}}(c,\cdot)\}_{r_{n}>|c|}$ is a normal
family. Since $H_{r_{n}}(c,\cdot)$ fixes $0,1,\infty$, there is a
subsequence of $\{ H_{r_{n}}(c,\cdot)\}$, which we still denote by
$\{ H_{r_{n}}(c,\cdot)\}$, that converges uniformly in the
spherical metric to a $(1+|c|)/(1-|c|)$-quasiconformal
homeomorphism $H(c,\cdot)$. Since $\tilde{H}(c,z_{i}) =H
(c,z_{i})$ for all $i=1,2,\cdots$, this implies that for any fixed
$c\in \Delta$, $H(c, z_{i})\neq H(c, z_{j})$ for $i\neq j$. Thus
$H(c, z): \Delta\times Z\to \C$ is a holomorphic motion.

For any $0<r<1$, $H(c,z)$ is $(1+r)/(1-r)$-quasiconformal for all
$c$ with $|c|\leq r$, it is $\alpha$-H\"older continuous, that is,
$$
d(H(c,z),H(c,z'))\leq A d(z,z')^{\alpha}\quad {\rm \ for \ all \ }
z,z' \in \C \quad {\rm \ and \  for  \ all \ } |c| \leq r,
$$
where $d(\cdot, \cdot)$ is the spherical distance and where $A$
and $0<\alpha<1$ depend only on $r$.

For any $z\in Z$ such that its spherical distances to $0$, $1$,
$\infty$ are greater than $\epsilon>0$, the map $H(c,z)$ is a
holomorphic map on $\Delta$, which omits the values $0$, $1$, and
$\infty$. So $H(c,z)$ decreases the hyperbolic distance
$\rho_{\Delta}$ on $\Delta$ and the hyperbolic distance
$\rho_{0,1}$ on $\C \setminus \{0,1,\infty\}$. So we have a
constant $B>0$ depending only on $r$ and $\epsilon$ such that
$$
d(H(c,z),H(c',z))\leq B |c-c'|
$$
for all $|c|, |c'|\leq r$ and all $z\in Z$ such that spherical
distances between them and $0$, $1$, and $\infty$ are greater than
$\epsilon>0$. Thus we get that
$$
d(H(c,z), h(c',z'))\leq A\delta(z,z')^{\alpha}+B |c-c'|.
$$
for $|c|,|c'|\leq r$ and $z,z' \in Z$ such that their spherical
distances from $0$, $1$, and $\infty$ are greater than
$\epsilon>0$. This implies that $H(c,z)$ is uniformly
equicontinuous on $|c|\leq r$ and $\{ z\in Z \;|\; d(z,
\{0,1,\infty\}) \geq \epsilon \}$. Therefore, its continuous
extension $H(c,z)$ is holomorphic in $c$ with $|c|\leq r$ for any
$\{ z\in \C \;|\; d(z, \{0,1,\infty\}\geq \epsilon \}$. Letting
$r\to 1$ and $\epsilon \to 0$, we get that $H(c,z)$ is holomorphic
in $c\in \Delta$ for any $z\in \C$. Thus $H (c,z): \Delta\times \C
\to \C$ is a holomorphic motion such that $H(c,z)|\Delta\times E=
h(c,z)$. We completed the proof.
\end{proof}

\section{The $|\epsilon\log \epsilon|$ continuity of a holomorphic motion}

In this section we show how the $|\epsilon \log \epsilon|$ modulus
of continuity for the tangent vector to a holomorphic motion can
be derived directly from Schwarz's lemma. Then we go on to show
how the H\"older continuity of the mapping $z \mapsto w(z)=h(c,z)$
with H\"older exponent $\frac{1-|c|}{1+|c|}$ follows from the
$|\epsilon \log \epsilon|$ continuity of the tangent vectors to
the curve $c \mapsto h(c,z)$. In particular, since any
$K$-quasiconformal map $z \mapsto f(z)$ coincides with $z \mapsto
h(c,z)$ where $K \leq \frac{1+|c|}{1-|c|}$, we conclude that $f$
satisfies a H\"older condition with exponent $1/K$.

\vspace*{10pt}
\begin{lemma}
Let $h(c,z)$ be a normalized holomorphic motion parametrized by
$\Delta$ and with base point $0$ and let $V(z)$ be the tangent
vector to this motion at $c=0$ defined by
\begin{equation}~\label{L11}
V(z)=\lim_{c\to 0} \frac{h(c,z)-z}{c}.
\end{equation}
Then $V(0)=0$,$V(1)=0$ and $|V(z)|=o(|z|^2)$ as $z
\rightarrow \infty.$
\end{lemma}

\begin{proof} Since $h(c,z)$ is normalized, $h(c,0)=0$ and
$h(c,1)=1$ for every $c\in \Delta$, and therefore $V(0)=0$ and
$V(1)=0$.  Since $h(c,\infty)=\infty$ for every $c\in \Delta$ if
we introduce the coordinate $w=1/z$ and consider the motion
$h_1(c,w)=1/h(c,1/w)$, we see that $h_1(c,0)=0$ for every $c\in
\Delta$.

Put $p(c)=h(c,z)$ and if we think of $z$ as a local coordinate for
the Riemann sphere,
$$
z \circ p(c)=z+c V^z(z)+o(c^2)
$$
and in terms of the local coordinate $w=1/z$,
$$
w \circ p(c)= w + c V^w(w)+o(c^2).
$$
Then $V^{w}(0)=0$. Putting $g=w \circ z^{-1}$, the identity
$g(z(p(c)))=w(p(c))$ yields
\begin{equation}~\label{transition}
g'(z(p(0))z'(p(0))=w'(p(0)).
\end{equation}
Since $g(z)=1/z,$ $g'(z)=-(1/z)^2$ and since $$V^{w} (0)=0, \
\frac{d}{dc}w(p(c))|_{c=0}=V^w(w(p(0)))$$ and $V^w(w(p(c))$ is a
continuous function of $c$, the equation
$$
V^z(z(p(c)))\frac{dw}{dz}=V^w(w(p(c)))
$$
implies
$$
\frac{V^z(z)}{z^2} \rightarrow 0
$$
as $z\rightarrow \infty$.
\end{proof}

Let $\rho_{0, 1}(z)$ be the infinitesimal form for the hyperbolic
metric on  $\C\setminus \{0, 1, \infty\}$ and let
$\rho_{\Delta}(z) =2/(1-|z|^{2})$ be the infinitesimal form for
the hyperbolic metric on $\Delta$. For any four distinct points
$a,b,c$ and $d$, the cross ratio
$$g(c)=cr(h_{c}(a),h_{c}(b),h_{c}(c),h_{c}(d))$$ is a holomorphic
function of $c\in \Delta$, and omitting the values $0,1$ and
$\infty$. Then by Schwarz's lemma,
$$
\rho_{0,1}(g(c))|g'(c)|\leq  \sigma_{\Delta} (c)
=\frac{2}{1-|c|^2}
$$
and
\begin{equation}~\label{eq1}
\rho_{0,1}(g(0))|g'(0)|\leq 2.
\end{equation}
But
\begin{equation}~\label{eq2}
|g'(0)|=|g(0)|
\left|\frac{V(b)-V(a)}{b-a}-\frac{V(c)-V(b)}{c-b}+\frac{V(d)-V(c)}{d-c}
-\frac{V(a)-V(d)}{a-d}\right|
\end{equation}
where $g(0)=cr (a,b,c,d)=\frac{(b-a)(d-c)}{(c-b)(a-d)}.$

\vspace*{10pt}
\begin{lemma}~\label{vanishing}
If $V(b)=o(b^2)$ as $b \rightarrow \infty$, then
$$
\left(\frac{V(b)-V(a)}{b-a}- \frac{V(c)-V(b)}{c-b}\right)
\rightarrow 0 \quad \hbox{as} \quad b \rightarrow \infty.
$$
\end{lemma}
\begin{proof}
$$\left(\frac{V(b)-V(a)}{b-a}- \frac{V(c)-V(b)}{c-b}\right)
$$
simplifies to
$$\frac{cV(b)-bV(c)-aV(b)-cV(a)+bV(a)+aV(c)}{(b-a)(c-b)}.$$
As $b \rightarrow \infty$ the denominator grows like $b^2$ but the
numerator is $o(b^2).$
\end{proof}

\vspace*{10pt}
\begin{theorem}~\label{Zygmundforvector} For any vector field $V$ tangent to a normalized
holomorphic motion and defined by (\ref{L11}), there exists a
number $C$ depending on $R$ such that for any two complex numbers
$z_1$ and $z_2$ with $|z_{1}|<R$ and $|z_{2}|<R$ and
$|z_1-z_2|<\delta,$
$$|V(z_{2})-V(z_{1})|\leq |z_{2}-z_{1}|(2+\frac{C}{\log \frac{1}{\delta}})(\log
\frac{1}{|z_{2}-z_{1}|}).$$
\end{theorem}

\begin{proof}
By applying Lemma \ref{vanishing}, inequality (\ref{eq1}) and
equation (\ref{eq2}) to $a=z_{1}, b=z_{2},c=0,d=\infty$, we obtain
$g(0)=\frac{z_{2}-z_{1}}{z_{2}},$
$$\left|\frac{V(b)-V(a)}{b-a}-\frac{V(c)-V(b)}{c-b}+\frac{V(d)-V(c)}{d-c}
-\frac{V(a)-V(d)}{a-d}\right|$$
$$=\left|\frac{V(z_{2})-V(z_{1})}{z_{2}-z_{1}}-
\frac{V(z_{2})}{z_{2}}\right|,$$ and
$$\rho_{0,1}\left(\frac{z_{2}-z_{1}}{z_{2}}\right)\left|\frac{z_{2}-z_{1}}{z_{2}}\right|
\left|\frac{V(z_{2})-V(z_{1})}{z_{2}-z_{1}}-\frac{V(z_{2})}{z_{2}}\right|\leq
2$$ and so
\begin{equation}\label{eq3}
\left|\frac{V(z_{2})-V(z_{1})}{z_{2}-z_{1}}-\frac{V(z_{2})}{z_{2}}\right|\leq
\frac{2}{\rho_{0,1}\left(\frac{z_{2}-z_{1}}{z_{2}}\right)
\left|\frac{z_{2}-z_{1}}{z_{2}}\right|}.
\end{equation}

Applying (\ref{eq1}) and (\ref{eq2}) again with $a=0, b=1,
c=\infty, d=z_{2}$,  we obtain
$$\rho_{0,1}(z_{2})|z_{2}|\left|\frac{V(z_{2})}{z_{2}}\right|\leq 2,$$
and so \begin{equation}\label{eq5} \frac{|V(z_{2})|}{|z_{2}|}\leq
\frac{2}{\rho_{0,1}(z_{2})|z_{2}|} \end{equation} and this
together with (\ref{eq3}) implies
\begin{equation}\label{eq4}\left|\frac{V(z_{2})-V(z_{1})}{z_{2}-z_{1}}\right|
\leq
\frac{2}{\rho_{0,1}\left(\frac{z_{2}-z_{1}}{z_{2}}\right)
\left|\frac{z_{2}-z_{1}}{z_{2}}\right|}
+\frac{2}{\rho_{0,1}(z_{2})|z_{2}|}.
\end{equation}

To finish the proof we need the following lemma, a form of which
appeared in~\cite[page 40]{LiZhongbook}. We adapted similar ideas
to prove the following version, which is sufficient for the proof
of Theorem~\ref{Zygmundforvector}.

\vspace*{10pt}
\begin{lemma}~\label{LiZhong}
If $0<|z|<1$, then
$$
\rho_{0,1}(z)\geq \frac{1}{|z|(\log r+\log\frac{1}{|z|})},
$$
where $r$ is chosen so that
$$
\log r > \max \{\frac{1}{\pi}\int \int_{\mathbb C} \frac{d \xi d
\eta}{|(\zeta+1)\zeta(\zeta-1)|}, 4+\log 4\}
$$
(Note that numerical calculation suggests that $4+\log 4$ is the
larger of these two numbers.)
\end{lemma}

\begin{proof}
From Agard's formula~\cite{Agard} (note that $\rho_{0,1}$ has the
curvature $-1$),
$$
\rho_{0,1}(z) = \left( \frac{1}{2\pi}\int \int_{\mathbb C}
\left|\frac{z(z-1)}{\zeta(\zeta-1)(\zeta-z)}\right| d \xi d \eta
\right)^{-1}.
$$
Since the smallest value of $\rho_{0,1}(z)$ on the circle $|z|=1$
occurs at  $z=-1$, we see that
$$
\frac{1}{\log r} \leq \min_{|z|=1} \rho_{0,1}(z) =
\Big(\frac{1}{\pi}\int \int_{\mathbb C} \left|\frac{1}{(\zeta-1)
(\zeta)( \zeta+1)}\right| d \xi d \eta\Big)^{-1}.
$$

The infinitesimal form of the Poincar\'e metric
$\rho_r=\rho_{\Delta_{r}^{*}}$ with curvature constantly equal to
$-1$ for the punctured disk $\Delta_r^{*}=\{ z\in \CP\;|\;
0<|z|<r\}$ is
\begin{equation}~\label{bla}
\rho_r(z)=\frac{1}{|z|\left[\log r + \log\frac{1}{|z|} \right]}.
\end{equation}
Note that $\rho_r(z)$ takes the constant value $\frac{1}{\log r}$
along $|z|=1$. Then
\begin{equation}\label{bd1} \rho (z) \leq
\rho_{0,1}(z) {\rm \ for \ all \ } z {\rm \ with \ } |z|= 1.
\end{equation}

Our next objective is to show that the same inequality
\begin{equation}\label{bd2} \rho (z) \leq
\rho_{0,1}(z)
\end{equation}
holds for all $z$ with $|z|<\delta$ when $\delta$ is sufficiently
small. In \cite[pages 17 and 18]{Ahlforsbook6} Ahlfors shows that
\begin{equation}\label{Ahlfors}
\rho_{0,1}(z)\geq \frac{|\zeta'(z)|}{|\zeta(z)|}\frac{1}{4+ \log
\frac{1}{|\zeta(z)|}}
\end{equation}
for $|z|\leq 1$ and $|z|\leq |z-1|,$  where $\zeta$  maps the
complement of $[1,+\infty]$ conformally onto the unit disk,
origins corresponding to each other and symmetry with respect to
the real axis being preserved. $\zeta$ satisfies
\begin{equation}\label{logzeta}
\frac{\zeta'(z)}{\zeta(z)}=\frac{1}{z\sqrt{1-z}},\ \end{equation}
\begin{equation}\label{Zeta1}\zeta(z)=\frac{\sqrt{1-z}-1}{\sqrt{1-z}+1}=
\frac{z}{(\sqrt{1-z}+1)^2}\end{equation}
 with ${\rm Re \ }
\sqrt{1-z}>0,$ and
\begin{equation}\label{Zeta2}|\zeta(z)| \rightarrow \frac{|z|}{4}
 \end{equation} as $z \rightarrow 0.$

We now show that there is $\delta>0$ such that if $|z|<\delta,$
then
$$
\frac{|\zeta'|}{|\zeta|}\frac{1}{[4+ \log
\frac{1}{|\zeta|}]}\geq \frac{1}{|z|[\log r+\frac{1}{|z|}]}.
$$
From (\ref{logzeta}) this is equivalent to showing that
$$
|\sqrt{1-z}|(4 + \log \frac{1}{|\zeta|}) \leq \log r+
\log\frac{1}{|z|},
$$
which is equivalent to
\begin{equation}\label{last}
|\sqrt{1-z}|(4+\log 4)\leq \log r+
\left\{\left(\log\frac{1}{|z|}\right)\left((1-|\sqrt{1-z}|)
\left(\frac{\log\frac{1}{|\zeta|}-\log4}
{\log\frac{1}{|z|}}\right)\right)\right\}.\end{equation} From
(\ref{Zeta2})
$$\left(\frac{\log\frac{1}{|\zeta|}-\log4}
{\log\frac{1}{|z|}}\right)$$ approaches $1$ as $z \rightarrow 0$
and the expression in the curly brackets on the right hand side of
(\ref{last}) approaches zero. Thus, in order to prove (\ref{bd2}),
it suffices to observe that
$$4 + \log 4 < \log r,$$
which is part of what we assumed.

We have so far established that $\rho_{0,1}(z)\geq \rho_R(z)$ on
the unit circle and on any circle $|z|=\delta$ for sufficiently
small $\delta.$  To complete the proof of the lemma we observe
that since both metrics $\rho_{0,1}(z)$ and $\rho_r(z)$ have
constant curvatures equal to $-1$, if we denote the Laplacian by
$$
\Delta = \left(\frac{\partial }{\partial x}\right)^2+
\left(\frac{\partial}{\partial y}\right)^2,
$$
then
$$
-\rho_{0,1}^{-2} \Delta \log \rho_{0,1} = - 1 {\rm \ and \ }
-\rho_r^{-2} \Delta \log \rho_r  = - 1.
$$
Therefore,
\begin{equation}~\label{curv}
\Delta (\log \rho_{0,1} - \log \rho_r) = \rho_{0,1}^2 - \rho_r^2
\end{equation}
throughout the annulus $\{z:\delta \leq |z| \leq 1\}.$  The
minimum of $\rho_{0,1}/\rho_r$ in this annulus
 occurs
either at a boundary point or in the interior. If it occurs at an
interior point, then its Laplacian of $\log(\rho_{0,1}/\rho_r \geq
1$ at that point and if it occurs on the boundary then
$\rho_{0,1}/\rho_r \geq 1$ at that point. In either case
$$ 0 \leq \Delta(\log \rho_{0,1} - \log
\rho_r)= \rho_{0,1}^2-\rho_r^{2}
$$
at that point, and therefore
$$
\rho_{0,1} \geq \rho_r
$$
throughout the annulus. This completes the proof of the lemma.
\end{proof}

From (\ref{eq4}) and this lemma we obtain
$$
|V(z_{2})-V(z_{1})|\leq
|z_{2}-z_{1}|\left(\left|\frac{V(z_{2})}{z_{2}}\right|+2 \log
r+2\log|z_{2}|+2\log \frac{1}{|z_{2}-z_{1}|}\right).
$$
Therefore to prove the theorem we must show that for
$\epsilon=\frac{C}{\log(1/\delta)}$
$$
\left|\frac{V(z_{2})}{z_{2}}\right|+2 \log r+2\log|z_{2}|+2\log
\frac{1}{|z_{2}-z_{1}|}\leq (2+\epsilon)\log
\frac{1}{|z_{2}-z_{1}|}.
$$
This is equivalent to showing that
$$
\left|\frac{V(z_{2})}{z_{2}}\right|+2 \log r+2\log|z_{2}|\leq
\epsilon \log \frac{1}{|z_{2}-z_{1}|}.
$$
If $|z_{2}|<1$, from (\ref{eq5}) and Lemma~\ref{LiZhong}, we have
$$
\rho_{0,1}(z_{2})\geq \frac{1}{|z_{2}| (\log
r+\log\frac{1}{|z_{2}|})},
$$
and
$$
\frac{|V(z_{2})|}{|z_{2}|}\leq 2\log r +2\log\frac{1}{|z_{2}|}.
$$
Hence
$$
\left|\frac{V(z_{2})}{z_{2}}\right|+2 \log r+2\log|z_{2}|\leq 4
\log r.
$$
If $1\leq|z_{2}|\leq R$, then since
$|\frac{V(z_{2})}{z_{2}}|+2\log|z_{2}|$ is a continuous function,
it is bounded by a number $M_{1},$ so
$$\left|\frac{V(z_{2})}{z_{2}}\right|
+2 \log r+2\log|z_{2}|\leq M_{1}+2 \log r.$$ The constant
$C=M_{1}+2 \log r$ does not depend on $\delta$ and
$\left|\frac{V(z_{2})}{z_{2}}\right|+2 \log r+2\log|z_{2}|\leq C$
for any $|z_{2}|\leq R.$ Thus, putting
$\epsilon=C/\frac{1}{\log(1/\delta)}$, we obtain
$$|V(z_{2})-V(z_{1})|\leq |z_{2}-z_{1}|(2+\epsilon)\left(\log
\frac{1}{|z_{2}-z_{1}|}\right).$$
\end{proof}
Applying the same argument at a variable value of $c$ we obtain
the following result.

\vspace*{10pt}
\begin{theorem}\label{concl}
Suppose $0<r<1$ and $R>0$. If $|c|\leq r$, $|z_{1}(c)|\leq R$,
$|z_{2}(c)| \leq R$ and $|z_{2}(c)-z_{1}(c)|<\delta$, then
\begin{equation}~\label{eq7}
|V(z_{2}(c))-V(z_{1}(c))|\leq
\frac{2+\epsilon}{1-|c|^2}|z_{2}(c)-z_{1}(c)|\log
\frac{1}{|z_{2}(c)-z_{1}(c)|},
\end{equation}
where $\epsilon \leq \frac{M}{\log (1/\delta)}$ and $\delta \geq
|z_1(0)-z_2(0)|$. Moreover, there is a constant $C$ such that
$$
|z_{2}(c)-z_{1}(c)|\leq C\cdot
|z_{2}-z_{1}|^{\frac{1-|c|}{1+|c|}}.$$
\end{theorem}

\begin{proof} Equation~(\ref{eq7}) follows by the same calculations we have just
completed.  To prove the second inequality, put
$s(c)=|z_{2}(c)-z_{1}(c)|$ and assume $0<|c|<1.$ Then~(\ref{eq7})
yields
$$
s'(c)\leq \frac{2+\epsilon}{1-|c|^2} s(c)\log\frac{1}{s(c)}.
$$
So
$$
-(\log\frac{1}{s(c)})'\leq
\frac{2+\epsilon}{1-|c|^2}\log\frac{1}{s(c)}
$$
and
$$
-(\log(\log\frac{1}{s(c)}))'\leq \frac{2+\epsilon}{1-|c|^2}.
$$
By integration,
$$
-\log(\log\frac{1}{s(c)})\Big|_{0}^{c}\leq -\frac{2+\epsilon}{2}
\log\frac{1-|c|}{1+|c|}\Big|_{0}^{|c|}
$$ and
$$
\log\log(\frac{1}{s(c)})-\log\log(\frac{1}{s(0)})
\geq\log\left(\frac{1-|c|}{1+|c|}\right)^{1+\frac{\epsilon}{2}}.
$$
Since $\log x$ is increasing,
$$
\frac{\log\frac{1}{s(c)}}{\log\frac{1}{s(0)}}\geq
\left(\frac{1-|c|}{1+|c|}\right)^{1+\frac{\epsilon}{2}},
$$
$$
\log s(c)\leq
\left(\frac{1-|c|}{1+|c|}\right)^{1+\frac{\epsilon}{2}} \log s(0)
$$
and
$$
s(c)\leq s(0)^{(\frac{1-|c|}{1+|c|})^{1+\frac{\epsilon}{2}}}.
$$
Putting $s=s(0)$ and $\alpha=\frac{1-|c|}{1+|c|}$, we wish to show
that
\begin{equation}~\label{final}
s^{\alpha^{1+\epsilon}}\leq C
s^\alpha {\rm \ or \ equivalently \ that \ }
s^{(\alpha^{1+\epsilon}-\alpha)} \leq C.
\end{equation}
This is equivalent to showing that
$$
\alpha(\alpha^{\epsilon} - 1)
\log s\leq \log C
$$
or that
$$
\alpha(\exp(\frac{M}{\log(1/s)}\log \alpha)-1) \log s \leq \log
C.
$$
Since $0<\alpha<1$ and since we may assume $s<e^{-1}$, by using
the inequality $e^x-1 \leq x e^{x_0} {\rm \ for \ } 0 \leq x \leq
x_0$, we see that it suffices to choose $C$ so that
$$
\alpha \frac{M}{\log(1/s)}\log (1/\alpha) e^{M \log \alpha}\log (1/s) =
\alpha M \log (1/\alpha) e^{M \log \alpha} \leq \log C.
$$
\end{proof}
The idea for the proof of Theorem \ref{concl} is suggested but not
worked out in  \cite{GardinerKeen1}.

\section{Kobayashi's metrics}

Suppose ${\mathcal N}$ is a connected complex manifold over a
complex Banach space. Let ${\mathcal H}={\mathcal H}(\Delta,
{\mathcal N})$ be the space of all holomorphic maps from $\Delta$
into ${\mathcal N}$. For $p$ and $q$ in ${\mathcal N}$, let
$$
d_{1}(p,q)=\log\frac{1+r}{1-r},
$$
where $r$ is the infimum of the nonnegative numbers $s$ for which
there exists $f\in {\mathcal H}$ such that $f(0)=p$ and $f(s)=q$.
If no such $f\in {\mathcal H}$ exists, then $d_{1}(p,q) =\infty$.

Let
$$
d_{n}(p,q)=\inf\sum_{i=1}^{n}d_{1}(p_{i-1},p_{i})
$$
where the infimum is taken over all chains of points $p_{0}=p,
p_{1},..., p_{n}=q$ in ${\mathcal N}$. Obviously, $d_{n+1}\leq
d_{n}$ for all $n>0$.

\vspace*{10pt}
\begin{definition}[Kobayashi's metric]
The Kobayashi pseudo-metric $d_{K}=d_{K, {\mathcal N}}$ is defined
as
$$
d_{K} (p,q)=\lim_{n\to\infty}d_{n}(p,q), \quad p, q\in {\mathcal
N}.
$$
\end{definition}

In general, it is possible that $d_K$ is identically equal to $0$,
which is the case for example if ${\mathcal N}=\CP$.

Another way to describe $d_{K}$ is the following. Let the
Poincar\'e metric on the unit disk $\Delta$ be given by
$$
\rho_{\Delta} (z,w) =  \log
\frac{1+\frac{|z-w|}{|1-\overline{z}w|}}{1-\frac{|z-w|}{|1-\overline{z}w|}},
\quad  z, w\in \Delta.
$$
Then $d_{K}$ is the largest pseudo metric on ${\mathcal N}$ such
that
$$
d_{K} (f(z), f(w)) \leq  \rho_{\Delta}$$ for all $z {\rm \ and \ }
w \in \Delta$ and for all holomorphic maps $f$ from $\Delta$ into
${\mathcal N}.$ The following is a consequence of this property.

\vspace*{10pt}
\begin{proposition}
Suppose ${\mathcal N}$ and ${\mathcal N}'$ are two complex
manifolds and $F: {\mathcal N}\to {\mathcal N}'$ is holomorphic.
Then
$$
d_{K, {\mathcal N}'} (F(p), F(q)) \leq d_{K, {\mathcal N}} (p,q).
$$
\end{proposition}

\vspace*{10pt}
\begin{lemma}~\label{Kg}
Suppose ${\mathcal B}$ is a complex Banach space with norm
$||\cdot||$. Let ${\mathcal N}$ be the unit ball of ${\mathcal B}$
and let $d_{K}$ be the Kobayashi's metric on ${\mathcal N}$. Then
$$
d_{K} (0, {\bf v}) =  \log \frac{1+||{\bf v}||}{1-||{\bf v}||} = 2
\tanh^{-1} ||{\bf v}||, \quad \forall \; {\bf v}\in {\mathcal N}.
$$
\end{lemma}

\begin{proof}
Pick a point ${\bf v}$ in ${\mathcal N}$. The linear function
$f(c)=c {\bf v}/||{\bf v}||$ maps the unit disk $\Delta$ into the
unit ball ${\mathcal N}$, and takes $||{\bf v}||$ into ${\bf v}$,
and $0$ into ${\bf 0}$. Therefore
$$
d_{K} (0,{\bf v})\leq \rho_{\Delta}(0,||{\bf v}||),
$$
where $\rho_{\Delta}$ is the Kobayashi's metric on $\Delta$ (it
coincides with the Poincar\'e metric on $\Delta$).

On the other hand, by the Hahn-Banach theorem, there exists a
continuous linear function $L$ on ${\mathcal N}$ such that $L({\bf
v})=||{\bf v}||$ and $||L||=1$. Thus, $L$ maps ${\mathcal N}$ into
the unit disk $\Delta$, and so
$$
d_{K} (0,{\bf v})\geq \rho_{\Delta}(0,||{\bf v}||).
$$
Therefore,
$$
d_{K}(0,{\bf v})=\rho_{\Delta}(0,||{\bf v}||)= \log \frac{1+||{\bf
v}||}{1-||{\bf v}||} = 2 \tanh^{-1}||{\bf v}||.
$$
\end{proof}

\section{\tes\ and Kobayashi's metrics on $T(R)$.}

Assume R is a Riemann surface conformal to $\Delta/\Gamma$ where
$\Gamma$ is a discontinuous, fixed point free group of hyperbolic
isometries of $\Delta$. Let ${\mathcal M}={\mathcal M}(\Gamma)$ be
the unit ball of the complex Banach space of all ${\mathcal
L}^{\infty}$ functions defined on $\Delta$ satisfying the
$\Gamma$-invariance property:
\begin{equation}~\label{Beltinvariance}
\mu(\gamma(z))\frac{\overline{\gamma'(z)}}{\gamma'(z)}=\mu(z)
\end{equation}
for all $z$ in $\Delta$ and all $\gamma$ in $\Gamma$. An element
$\mu\in {\mathcal M}$ is called a Beltrami coefficient on $R$.
Points of the Teichm\"uller space $T=T(R)$ are represented by
equivalence classes of Beltrami coefficients $\mu\in {\mathcal
M}$. Two Beltrami coefficients $\mu, \nu\in {\mathcal M}$ are in
the same \te\ equivalence class if the quasiconformal self maps
$f^{\mu}$ and $f^{\nu}$ which preserve $\Delta$ and which are
normalized to fix $0,i$ and $-1$ on the boundary of the unit disk
coincide at all boundary points of the unit disk.

\vspace*{10pt}
\begin{definition}[\tes\ metric]
For two elements $[\mu]$ and $[\nu]$ of $T(R)$, \tes\  metric is
equal to
$$
d_{T}([\mu],[\nu])=\inf \log K(f^{\mu}\circ (f^{\nu})^{-1}),
$$
where the infimum is over all $\mu$ and $\nu$ in the equivalence
classes $[\mu]$ and $[\nu]$, respectively. In particular,
$$
d_{T}(0,[\mu])=\log \frac{1+k_{0}}{1-k_{0}}
$$
where $k_{0}$ is the minimal value of $||\mu||_{\infty}$, where
$\mu$ ranges over the \te\ class $[\mu]$.
\end{definition}

\vspace*{10pt}
\begin{lemma}\label{easy}
Let $d_{K}$ and $d_{T}$ be Kobayashi's and \tes\ metrics of
$T(R)$.  Then $d_{K}\leq d_{T}$.
\end{lemma}

\begin{proof}
Let a Beltrami coefficient $\mu$ satisfying (\ref{Beltinvariance})
be extremal in its class and $||\mu||_{\infty}=k$. This is
possible because by normal families argument every class possesses
at least one extremal representative. By the definition of \tes\
metric
$$
d_{T}(0,[\mu])=\log \frac{1+k}{1-k}.
$$
For such a $\mu$, let $g(c)=[c\mu/k]$. Then $g(c)$ is a
holomorphic function of $c$ for $|c|<1$ with values in the \te\
space $T(R)$, $g(0)=0$ and $g(k)=[\mu]$. Hence
$$
d_{K}(0, [\mu])\leq d_{1}(0, [\mu])\leq d_{T}(0,[\mu]).
$$
Now the right translation mapping $\alpha([f^{\mu}])=[f^{\mu}\circ
(f^{\nu})^{-1}]$ is biholomorphic, so it is an isometry in
Kobayashi's metric. We also know that it is an isometry in
Teichmuller's metric. Therefore, the inequality
$$
d_{K}([\nu],[\mu])\leq d_{1}([\nu],[\mu])\leq
d_{T}([\nu],[\mu])
$$
holds for an arbitrary pair of points $[\mu]$
and $[\nu]$ in the \te\ space $T(R)$.
\end{proof}

In order to describe holomorphic maps into $T(R)$ we will use the
Bers' embedding by which $T(R)$ is realized as a bounded domain in
the Banach space ${\mathcal B}(R)$ of equivariant cusp forms. Here
${\mathcal B}(R)$ consists of the functions $\varphi$ holomorphic
in $\Delta^c$ for which
$$
\sup_{z \in \Delta^c}\{|(|z|^2-1)^2|\varphi(z)|\}<\infty
$$
and for which
$$
\varphi(\gamma(z))(\gamma'(z))^2 = \varphi(z) {\rm \ for \ all \ } \gamma \in \Gamma.
$$

We assume $\Gamma$ is a Fuchsian covering group such that
$\Delta/\Gamma$ is conformal to $R$. For any Beltrami differential
$\mu$ supported on $\Delta$, we let $w^{\mu}$ be the
quasiconformal self-mapping of $\overline{\mathbb C}$ which fixes
$1,i$ and $-1$ and which has Beltrami coefficient $\mu$ in
$\Delta$ and Beltrami coefficient identically equal to zero in
$\Delta^c$. Let $w^{\mu}$ restricted to $\Delta^c$ be equal to the
Riemann mapping $g^{\mu}$. Then $g^{\mu}$ has the following
properties:
\begin{itemize}
\item[a)] $g^{\mu}$ fixes the points $1, i$ and $-1$,
\item[b)] $g^{\mu}(\partial \Delta)$ is a quasiconformal image of the
circle $\partial \Delta$,
\item[c)] $g^{\mu}$ is univalent and holomorphic in $\Delta^c$.
\item[d)] $g^{\mu}\circ \gamma \circ (g^{\mu})^{-1}$ is equal to a M\"obius
transformation $\tilde{\gamma}$, for all $\gamma$ in $\Gamma$, and
\item[e)] $g^{\mu}$ determines and is determined uniquely by the
corresponding point in $T(R)$.
\end{itemize}

The Bers' embedding maps the \te\ equivalence class of $\mu$ to
the Schwarzian derivative of $g^{\mu}$ where the Schwarzian
derivative of a $C^3$ function $g$ is defined by
$$
S(g)=\left(\frac{g''}{g'}\right)'  + \frac{1}{2}
\left(\frac{g''}{g'}\right)^2.
$$

In the next section we use this realization of the complex
structures to prove that $d_T \leq d_K$.

\section{The Lifting Problem}

Let $\Phi$ be the natural map from the space ${\mathcal M}$ of
Beltrami differentials on $R$ onto $T(R)$ and let $f$ be a
holomorphic map from the unit disk into $T(R)$ with $f(0)$ equal
to the base point of $T(R)$. The lifting problem is the problem of
finding a holomorphic map $\tilde{f}$ from $\Delta$ into
${\mathcal M}$, such that $\tilde{f}(0)=0$ and $\Phi \circ
\tilde{f} = f$.

In this section we  prove the theorem of Earle, Kra and
Krushkal~\cite{EKK} which says that the lifting problem always has
a solution. We follow their technique which relies on proving an
equivariant version of Slodkowski's extension theorem and then
going on to show that the positive solution to the lifting problem
implies $d_T \leq d_K$ for every Riemann surface that has a
nontrivial \te\ space with complex structure.

\vspace*{10pt}
\begin{theorem}[An equivariant version of Slodkowski's extension theorem]
Let $h(c,z)$ be a holomorphic motion of
$\Delta^c=\overline{\mathbb C} \setminus \Delta$ parametrized by
$\Delta$ and with base point $0$ and let $\Gamma$ be a
torsion-free group of M\"obius transformations mapping $\Delta^c$
onto itself. Suppose for each $\gamma\in \Gamma$ and $c\in\Delta$
there is a M\"obius transformation $\tilde{\gamma}_{c}$ such that
$$
h(c,\gamma(z))=\tilde{\gamma}_{c}(h(c,z)), \;\; \forall \; z\in
\Delta^{c}.
$$
Then $h(c,z)$ can be extended to a holomorphic motion $H(c,z)$ of
$\C$ parametrized by $\Delta$ and with base point $0$ in such a
way that
$$
H(c,\gamma(z))=\tilde{\gamma}_{c}(H(c,z))
$$
holds for $\gamma\in\Gamma$, $c\in\Delta$ and $z\in \C$.
\end{theorem}

\begin{proof}
Observe that $\tilde{\gamma}_{c}$ is uniquely determined for all
$c\in\Delta$ because $\Delta^c$ contains more than two points. To
extend $h(c,z)$ to $\Delta$, start with an point $w\in \Delta$. By
Theorem~\ref{hmt}, the motion $h(c,z)$ can be extended to a
holomorphic motion (still denote it as $h(c,z)$) of the closed set
$\Delta^{c}\cup \{w\}$. Furthermore, we may extend it to the orbit
of $w$ using the $\Gamma$-invariant property:
$$
h(t,\gamma(w))=\tilde{\gamma}_{c}(h(c,w)),
$$ for all
$\gamma\in\Gamma$. Since every $\gamma\in\Gamma$ is fixed point
free on $\Delta$, the motion $h(c,z)$ is well defined and
satisfies the $\Gamma$-invariant property for all $c\in\Delta$ and
all $z$ in the set
$$
E=\{\gamma(w):\gamma\in\Gamma\}\cup (\C\setminus \Delta).
$$

So we only need to show that $h(c,z)$ is a holomorphic motion of
$E$. Observe first that $h(0,z)=z$ since
$\tilde{\gamma}_{0}=\gamma$ for all $\gamma\in\Gamma$. To show
$h(c,z)$ is injective for all fixed $c\in \Delta$, suppose
$h(c,z_{1})=h(c,z_{2})$ for some $c\in\Delta$. Since $h(c,z)$ is
injective on $\Delta^{c}\cup \{w\}$, we may assume that
$z_{1}=g(w)$ for some $g\in \Gamma$. By $\Gamma$-invariant
property,
$$
h(c,w)=(\tilde{g}_{c})^{-1}(h(c,z_{1})).
$$
Thus,
$$
h(c,w)=(\tilde{g}_{c})^{-1}(h(c,z_{2}))=h(c,g^{-1}(z_{2})),
$$
and we conclude that $z_{2}$ belongs to the $\Gamma$-orbit of $w$.
Let $z_{2}=\beta(w)$ for some $\beta\in\Gamma$. Then
$$
h(c,w)=\tilde{\gamma}_{c}(h(c,w)),
$$
where $\gamma=g^{-1}\circ\beta$. Therefore $h(c,w)$ is a fixed
point of $\tilde{\gamma}_{c}$. On the other hand, since $\gamma$
is a hyperbolic M\"obius transformation, $\tilde{\gamma}_{c}$ is
also hyperbolic, so unless $\tilde{\gamma}_{c}$ is identity, it
can only fix points on the set $h(c,\partial \Delta))$. Hence
$\gamma$ is the identity map and $z_{1}=z_{2}$.

Finally, we will show that $l:c\to h(c,z)$ is holomorphic for any
fixed $z\in E$. we may assume $z=g(w)$, $g\in\Gamma\setminus
\{identity\}$. Then $l(c)=h(c,g(w))=\tilde{g}_{c}(h(c,w))$. Since
$c\to h(c,w)$ is holomorphic and $\tilde{g}_{c}$ is a M\"obius
transformation, it is enough to prove the map $k:c\to
\tilde{g}_{c}(\zeta)$ is holomorphic for any fixed $\zeta$.
Applying the $\Gamma$-invariant property to the three points
$0,1,\infty$, we obtain
$$
\tilde{g}_{c}(0)=h(c,g(0)),
$$
$$
\tilde{g}_{c}(1)=h(c,g(1)),
$$
$$
\tilde{g}_{c}(\infty)=h(c,g(\infty)).
$$
The right-hand sides of these three equations are holomorphic, so
the maps $c\mapsto\tilde{g}_{c}(0)$, $c\mapsto\tilde{g}_{c}(1)$
and $c\mapsto\tilde{g}_{c}(\infty)$ are holomorphic. Since
$\tilde{g}_{c}$ is a M\"obius transformation, $k: c\to
\tilde{g}_{c}(\zeta)$ is holomorphic.

Therefore, we have extended $h(c,z)$ to a holomorphic motion of
$$\Delta^c \cup \{the \ \Gamma \ orbit \ of \ z\}.$$ By repeating
this extension process to a countable set of points whose $\Gamma$
orbits are dense in $\Delta,$ we obtain the extension $H(c,z)$ of
$h(c,z)$ with the property that
$$
H(c,\gamma(z))=\tilde{\gamma}_{c}(H(c,z))
$$
for all $\gamma\in\Gamma$, $c\in\Delta$ and $z\in \C.$
\end{proof}

This equivariant version of Slodkowski's extension theorem leads
almost immediately to the following lifting theorem.

\vspace*{10pt}
\begin{theorem}[The lifting theorem]
If $f:\Delta\to T(R)$ is holomorphic, then there exists a
holomorphic map $\tilde{f} : \Delta \to {\mathcal M}$ such that
$$
\Phi\circ \tilde{f}=f.
$$
If $\mu_{0}\in {\mathcal M}$ and $\Phi(\mu_0)=f(0)$, we can choose
$\tilde{f}$ such that $\tilde{f}(0)=\mu_{0}$.
\end{theorem}

\begin{proof}
By using the translation mapping $\alpha$ of the \te\ space given
by
$$
\alpha([w^{\mu}])=[w^{\mu}\circ (w^{\nu})^{-1}],
$$
we may assume $f(0)=0$. For each $c\in\Delta$, let $g(c,\cdot)$ be
a meromorphic function whose Schwarzian derivative is $f(c)$. Then
on $\C\setminus \Delta$ the map $g(c,\cdot)$ is injective, and we
can specify $g(c,\cdot)$ uniquely by requiring that it fix $1,i$
and $-1.$ Thus $g(0,z)=z$. It is easy to verify that
$$
g(c,z):\Delta \times (\C\setminus \Delta)\to \C
$$
is a holomorphic motion. For every $\gamma\in\Gamma$ and
$c\in\Delta$, there exists a M\"obius transformation
$\tilde{\gamma}_{c}$ such that
$$
g(c,\gamma(z))=\tilde{\gamma}_{c}(g(c,z)).
$$
Using the equivalent version of Slodkowski's extension theorem, we
extend $g$ to a $\Gamma$-invariant holomorphic motion (still
denote it as $g$) of $\C$. For each $c\in\Delta$, let
$\tilde{f}(c)$ be the complex dilatation
$$
\tilde{f}(c)=\frac{g_{\overline{z}}}{g_{z}}.
$$
Then the $\Gamma$-invariant property of $g$ implies that
$\tilde{f}(c)\in {\mathcal M}$. From Theorem~\ref{hmt} in Section
4, we know that $\tilde{f}(c)$ is a holomorphic function of $c$.
By the definition of the Bers embedding, $\Phi(\tilde{f}(c))$ is
the Schwarzian derivative of $g$. So $\Phi(\tilde{f}(c))=g(c)$.
\end{proof}

Now we will use the lifting theorem to show that the Teichm\'uller
metric and Kobayashi's metric of $T(R)$ coincide.

\vspace*{10pt}
\begin{lemma}~\label{balllemma}
Suppose ${\mathcal M}$ is the unit ball in the space of
essentially bounded Beltrami differentials on a Riemann surface
$R$. Let $d_{K}$ be the Kobayashi's metric on ${\mathcal M}$. Then
$$
d_{K}(\mu,\nu)= 2
\tanh^{-1}\left\|\frac{\mu-\nu}{1-\overline{\nu}\mu}\right\|_{\infty}
$$
for all $\mu$ and $\nu$ in ${\mathcal M}$.
\end{lemma}

\begin{proof}
From Lemma~\ref{Kg}, for any $\nu\in {\mathcal M}$,
$$
d_{K} (0, \nu) = 2 \tanh^{-1}\left\|\nu\right\|_{\infty}
$$
Observe the function defined by
$$
\lambda \to \frac{\nu-\lambda}{1-\overline{\nu}\lambda}
$$
is a biholomorphic self map of ${\mathcal M}$. Therefore
$$
d_{K}(\mu,\nu)=2 \tanh^{-1}\left\|
\frac{\mu-\nu}{1-\overline{\nu}\mu}\right\|_{\infty}.
$$
\end{proof}

\vspace*{10pt}
\begin{theorem}[\cite{Royden},~\cite{Gardiner3},~\cite{Gardinerbook}]
The \tes\ and Kobayashi's metrics of $T(R)$ coincide.
\end{theorem}

\begin{proof}
In Lemma~\ref{easy} we already showed that $d_{K}\leq d_{T}$, So
we only need to prove $d_{K}\geq d_{T}$. Choose a holomorphic map
$f:\Delta\to T(R)$ so that $f(0)=0$ and $f(c)=[\mu]$ for some
$c\in \Delta$. Then the lifting theorem implies there exists a
holomorphic map $\tilde{f}:\Delta\to {\mathcal M}$ so that
$$
\Phi(\tilde{f}(c))=f(c)=[\mu].
$$
So
$$
d_{K} (0,\tilde{f}(c))\leq \rho_{\Delta}(0,c).
$$
By Lemma~\ref{balllemma} and definition of Teichm\"uller metric,
$$
d_{T}(0,[\mu])\leq d_{K} (0,\tilde{f}(c)).
$$
Therefore,
$$
d_{T}(0,[\mu])\leq \rho_{\Delta}(0,c).
$$
Taking the infimum over all such $f$,
we have
$$
d_{T}(0,[\mu])\leq d_{K}(0,[\mu]).
$$
Hence $d_{T}\leq d_{K}$.
\end{proof}

\section{holomorphic motions and the Fatou linearization theorem.}

In this  section we apply holomorphic motions to the Fatou
linearization theorem,~\cite[\S 3]{Jiang3}. For more applications
of holomorphic motions to complex dynamics, we refer the reader
to~\cite{Jiang3},~\cite{Jiang2} and ~\cite{Jiang1}.

\subsection{Parabolic germs.}

Suppose $f(z)$ is a parabolic germ at $0$. This means that after a
change of coordinates, there is a constant $0<r_{0}<1/2$ such that
$f(z)$ is defined in the disk $\Delta_{r_{0}}$ and is conformal
with the Taylor expansion
$$
f(z) = e^{2\pi p i\over q} z+ higher \ order \ terms,  \quad
(p,q)=1.
$$
We also assume $f^{m}\not\equiv id$ for all $m>0,$ and this
assumption implies (see \cite{Milnorbook})
$$
f^{q}(z) = z (1+ a z^{n} + \epsilon(z)), \quad a\neq 0, \quad
|z|<r_{0},
$$
where $n$ is a multiple of $q$ and $\epsilon(z)$ is given be a
convergent power series of the form
$$\epsilon(z)=a_{n+1}z^{n+1}+a_{n+2} z^{n+2} + \cdots.$$

 Suppose $0<r<r_{0}$. A simply
connected open set ${\mathcal P}\subset \Delta_{r}\cap f^{q}
(\Delta_{r})$ satisfying $f^{q}({\mathcal P}) \subset {\mathcal
P}$ and $0\in \overline{\mathcal P}$ is called an attracting petal
for $f$ if $(f^{q})^{m}(z)$ for $z\in {\mathcal P}$ converges
uniformly to $0$ as $m\to \infty$. An attracting petal ${\mathcal
P}'$ for $f^{-1}$ is called a repelling petal. The Leau-Fatou
flower theorem~\cite{ Fatou1,Fatou2,Fatou3,Leau} says that there
exist $n$ attracting petals $\{ {\mathcal P}_{i}\}_{i=0}^{n-1}$
and $n$ repelling petals $\{ {\mathcal P}_{j}'\}_{j=0}^{n-1}$ such
that
$$
N =\cup_{i=0}^{n-1}{\mathcal P}_{i}\cup \cup_{j=0}^{n-1} {\mathcal
P}_{j}'\cup \{0\}
$$
is a neighborhood of $0$. For an exposition of this result see
~\cite{CarlesonGamelinbook} or~\cite{Milnorbook}.

\subsection{Fatou linearization.}

For each attracting petal ${\mathcal P}={\mathcal P}_{i}$,
consider the change of coordinate on ${\mathcal P}$:
$$
w= \phi(z) =\frac{d}{z^{n}}, \quad d=-\frac{1}{na}, \;\; z\in
{\mathcal P}.
$$
Suppose the image of ${\mathcal P}$ by $\phi$ is the right
half-plane
$$
R_{\tau} =\{ w\in {\mathbb C}\;\;|\;\; \Re{w} > \tau\}.
$$
Then
$$
z=\phi^{-1} (w) = \sqrt[n]{\frac{d}w}: R_{\tau} \to {\mathcal P}
$$
is a conformal map. The conjugate of $f^{q}$ by $\phi$ on
$R_{\tau}$ is
$$
F(w) = \phi\circ f^q \circ \phi^{-1}(w) =w+1 +\eta\Big(
\frac{1}{\sqrt[n]{w}}\Big)\quad \hbox{as} \quad w\in R_{\tau}
$$
where $\eta(\xi)$ is an analytic function in a neighborhood of $0$
given by a convergent power series of the form
$$
\eta(\xi) =b_{1}\xi+b_{2}\xi^{2}+\cdots, \quad |\xi|\leq r.
$$
Take $r$ small enough so that
$$
|\eta (\xi)| \leq \frac{1}{2}, \quad \forall \; |\xi|\leq r.
$$
Then $F(R_{\tau}) \subset R_{\tau}$ for $\tau \geq
\frac{1}{r^{n}}$ because if $\Re w \geq \tau \geq
\frac{1}{r^{n}},$ then $\left|\frac{1}{\sqrt[n]{w}}\right|<r$ and
$$
\Re{F(w)} =\Re{w} +1 + \Re{\eta \Big(\frac{1}{\sqrt[n]{w}}\Big)}
\geq \tau +1 -|\eta\Big(\frac{1}{\sqrt[n]{w}}\Big)| \geq \tau
+\frac{1}{2}, \quad \forall \; w\in R_{\tau}.
$$

\vspace*{10pt}
\begin{theorem}[Fatou Linearization Theorem]~\label{flt}
Suppose $\tau > 1/r^{n}+1$ is a real number. Then there is a
simply connected domain $\Omega$ and a conformal map $\Psi$ from
$R_{\tau}$ onto $\Omega$ such that
$$
F (\Psi (w)) = \Psi (w+1), \quad \forall w\in R_{\tau}.
$$
\end{theorem}

\vspace*{10pt} For an exposition of this result see
~\cite{CarlesonGamelinbook} or~\cite{Milnorbook}. Here, we give an
alternative exposition based on holomorphic motions .

\subsection{Construction of a holomorphic motion.}
For any $x\geq \tau$, consider the vertical lines with infinity
attached
$$
E_{0,x} = \{ w\in {\mathbb C}\;|\; \Re{w} =x\}
$$
and
$$
E_{1,x} = \{ w\in {\mathbb C}\;|\; \Re{w} =x+1\}
$$
and let $E_{x}$ be the set
$$
E_{0,x}\cup E_{1,x}.
$$
Define $H_{x} (w): E_{x}\to \hat{\mathbb C}$ by
$$
H_{x} (w) =\left\{
\begin{array}{ll}
        w, & w\in E_{0,x}; \cr
        w +
\eta\Big(\frac{1}{\sqrt[n]{w-1}}\Big), & w\in E_{1,x}.
\end{array}\right.
$$
Since $H_{x}(w)$ is injective separately on $E_{0,x}$ and on
$E_{1,x},$ and since for $w \in E_{1,x},$
$$
\Re(H_{x}(w)) \geq \Re(w) -\frac{1}{2} =x+1-\frac{1}{2} =
x+\frac{1}{2},
$$
$H_{x}(w)$  is injective on $E_{x}.$ Moreover, $H_{x}(w)$
conjugates $F(w)$ to the linear map $w\mapsto w+1$ on $E_{0,x}$,
that is,
$$
F (H_{x}(w)) =H_{x} (w+1), \quad \forall\; w\in E_{0,x}.
$$

To obtain a holomorphic motion, we  introduce a complex parameter
$c\in \Delta$ into $\eta(\xi)$ as follows.
$$
\eta (c, \xi) = \eta (c r \xi\sqrt[n]{x-1}) = b_{1} (cr
\xi\sqrt[n]{x-1}) +b_{2} (cr\xi\sqrt[n]{x-1})^{2} +\cdots, \;\;
|\xi| \leq \frac{1}{\sqrt[n]{x-1}}.
$$
Then
$$
|\eta (c, \xi)| \leq \frac{1}{2}, \quad \forall \; |c|<1, \; |\xi|
\leq \frac{1}{\sqrt[n]{x-1}}.
$$
Now we define
$$
H_{x} (c, w) =\left\{
\begin{array}{ll}
        w, & (c, w)\in \Delta\times E_{0,x}; \cr
        w +
\eta\Big(c, \frac{1}{\sqrt[n]{w-1}}\Big), & (c, w)\in \Delta\times
E_{1,x}.
\end{array}\right.
$$
\begin{lemma}\label{hmt2}
$$
H_{x}(c,w):  \Delta\times E_{x}\to \hat{\mathbb
C}
$$
is a holomorphic motion.
\end{lemma}
\begin{proof}

\indent (1) Clearly, $H_{x}(0, w)=w$ for all $w\in E_{x}$.

\indent (2) By Rouch\'e's theorem, for any fixed $c \in \Delta$,
$H_{x}(w)$ is injective on $E_{0,x}$ and on $E_{1,x}.$ Moreover,
since for $w \in E_{1,x},$
$$
\Re{H_{x}(c, w)} =\Re{w} +\Re{\eta\Big(c,
\frac{1}{\sqrt[n]{w-1}}\Big)} \geq \Re{w} -\frac{1}{2},
$$
the images of $E_{0,x}$ and $E_{1,x}$ by $H_{x}(c, \cdot)$ are
disjoint. Thus $H_{x}(c,w)$ is injective on $E_{x}$.

(3) For any fixed point $w\in E_{0,x}$, $H_{x}(c,w) =w$ and for
any point $w\in E_{1,x}$,
$$
H_{x} (c, w) = w + \eta\Big(c, \frac{1}{\sqrt[n]{w-1}}\Big).
$$
But $\eta\Big(c, \frac{1}{\sqrt[n]{w-1}}\Big)=\eta (c r
\xi\sqrt[n]{x-1})$ is a convergent power series and thus a
holomorphic function of $c\in \Delta$. Therefore
$$
H_{x}(c,w): \Delta\times E_{x}\to \C
$$
is a holomorphic motion. \end{proof}

By Theorem~\ref{hmt}, $H_{x}(c,w)$ can be extended to a
holomorphic motion of the entire plane; we still denote the
extension by the same symbols,
$$
H_{x} (c, w): \Delta \times \C \to \C.
$$
Also by Theorem~\ref{hmt}, the  map
$$
w \mapsto H_{x} (c,w)
$$
is $(1+|c|)/(1-|c|)$-quasiconformal for $|c|<1.$ In particular, if
$c=1/(r\sqrt[n]{x-1})$, then $w \mapsto H_{x} (c,w)$ is a
quasiconformal with dilatation is less than or equal to
$$
K(x)
=\frac{1+\frac{1}{r\sqrt[n]{x-1}}}{1-\frac{1}{r\sqrt[n]{x-1}}},
$$
 $K(x) \to 1$ as $x\to \infty.$ This observation will be
important later when we convert from a quasiconformal conjugacy to
a conformal conjugacy.

\subsection{Construction of a quasiconformal conjugacy.}

Let
$$
S_{x} =\{ w\in \CP\;|\; x\leq \Re{w}\leq x+1\}
$$
be the strip bounded by two lines $\Re{w}=x$ and $\Re{w}=x+1$. Let
$h_{c(x)}(w)=H_x(c,w)$ for $w$ in the strip $S_x.$

For any $w_{0}\in R_{\tau}\cup E_{0,\tau}$, let $w_{m}
=F^{m}(w_{0})$. Since $w_{m+1}-w_{m}$ tends to $1$  uniformly on
$R_{\tau}\cup E_{0,\tau}$ as $m$ goes to $\infty$,
$$
\frac{w_{m}-w_{0}}{m} =\frac{1}{m} \sum_{k=1}^{m} (w_{k}-w_{k-1})
\to 1
$$
uniformly on $R_{\tau}\cup E_{0,\tau}$ as $m$ goes to $\infty$. So
$w_{m}$ is asymptotic to $m$ as $m$ goes to $\infty$ uniformly in
any bounded set of $R_{\tau}\cup E_{0,\tau}.$  In particular, if
$x_{0}=\tau$ and $\xi_{m} =F^{m}(x_{0})$ and $x_{m}=\Re(\xi_{m})$.
Then $x_{m}$ is asymptotic to $m$ as $m$ goes to $\infty$.

For each $m>0$, the set
$$
\Upsilon_{m} = F^{-m}(E_{0, x_{m}})
$$
is a curve passing passes through $x_{0}=\tau$ and $\infty,$ and
if
$$
\Omega_{m} =F^{-m} (R_{x_{m}}),
$$
and $\Upsilon_{m}$ is the boundary of $\Omega_m.$

Let
$$
S_{i, x_{m}} =F^{-i}(S_{x_{m}}) {\rm \ for \ } i=m, m-1, \cdots,
1, 0, -1, \cdots, -m+1, -m, \cdots
$$
and
$$
\Omega_{m} =\cup^{i=m}_{-\infty} S_{i,x_{m}}.
$$
Similarly, let
$$
A_{m}= \{ w\in {\mathbb C}\;\;|\;\; \tau+ m\leq \Re{w} \leq
\tau+m+1\}
$$
and let
$$
A_{i, m} =\{ w\in {\mathbb C}\;\;|\;\; \tau+ m-i\leq \Re{w} \leq
\tau+m+1-i\}
$$
for $i=m, m+1, \cdots, 1, 0, -1, \cdots, -m+1, -m, \cdots$.

Then
$$
\beta_{m} (w) =w+ x_{m}-\tau -m: {\mathbb C}\to {\mathbb C}.
$$
is conformal and
$$
h_{c(x_{m})}\circ \beta_{m} (A_{m}) =S_{x_{m}}.
$$
is a $K(x_{m})$-quasiconformal homeomorphism on $A_{m}$. Moreover,
$$
F(\psi_{m} (w))=\psi_{m} (w+1), \quad  \forall \; \Re{w}=m+\tau.
$$
Furthermore, if we define
$$
\psi_{m}(w) = F^{-i}(\psi_{m} (w+i)), \quad \forall \; w\in
A_{i,m}
$$
for $i=-m, -m+1, \cdots, -1, 0, 1, \cdots, m-1, m, \cdots,$ then
it is $K(x_{m})$-quasiconformal homeomorphism from $R_{\tau}$ to
$\Omega_{m}$ and
$$
F (\psi_{m}(w)) = \psi_{m}(w+1), \quad \forall \; w\in R_{\tau}.
$$

\subsection{Improvement to conformal conjugacy.}

Let $w_{0}=\tau$ and $w_{m} = F^{m}(w_{0})$ for $m=1, 2,\cdots$.
Remember that
$$
R_{x_{m}}=\{ w\in \CP\;|\; \Re{w} >x_{m}\}
$$
where $x_{m}=\Re{w_{m}}$.

For any $\tilde{w}_{0}\in R_{x_{m+1}}$, let $\tilde{w}_{m} =
F^{m}(\tilde{w}_{0})$ for $m=1, 2, \cdots$. Since
$$
F'(w) = 1+O\Big(\frac{1}{|w|^{1+\frac{1}{n}}}\Big),\quad w\in
R_{\tau}
$$
and $\tilde{w}_{m}/m \to 1$ as $m\to \infty$ uniformly on any
compact set, there is a constant $C>0$ such that
$$
C^{-1}\leq \frac{|\tilde{w}_{m}-w_{m}|}{|\tilde{w}_{1}-w_{1}|}
=\prod_{k=1}^{m}
\frac{|\tilde{w}_{k+1}-w_{k+1}|}{|\tilde{w}_{k}-w_{k}|} =
\prod_{k=1}^{m} \Big(1+O\Big(\frac{1}{k^{1+\frac{1}{n}}}\Big)\Big)
\leq C
$$
as long as $w_{1}$ and $\tilde{w_{1}}$ keep in a same compact set.
Since
$$
w_{m+1} =w_{m} +1+\eta\Big(\frac{1}{\sqrt[n]{w_{m}}}\Big) \quad
\hbox{and}\quad |\eta\Big(\frac{1}{\sqrt[n]{w_{m}}}\Big)|\leq
\frac{1}{2},
$$
the distance between $w_{m+1}$ and $R_{x_{m}}$ is greater than or
equal to $1/2$. So the disk $\Delta_{1/2}(w_{m+1})$ is contained
in $R_{x_{m}}$. This implies that the disk
$\Delta_{1/(2C)}(w_{1})$ is contained in $\Omega_{m}$ for every
$m=0, 1, \cdots$. Thus the sequence
$$
\psi_{m} (w) : R_{\tau}\to \Omega_{m}, \quad m=1, 2, \cdots
$$
is contained in a weakly compact subset of the space of
quasiconformal mappings. Let
$$
\Psi (w): R_{\tau}\to \Omega
$$
be a limiting mapping of a subsequence. Then $\Psi $ is
$1$-quasiconformal and thus conformal and satisfies
$$
F (\Psi (w)) = \Psi (w+1), \quad \forall w\in R_{\tau}.
$$
This completes the proof of Theorem~\ref{flt}.

\nocite{GardinerLakicbook}

\vspace*{20pt}

\bibliographystyle{plain}
\bibliography{../articles,../books}

\noindent {Frederick P. Gardiner}, Department of Mathematics,
Brooklyn College, Brooklyn, NY 11210 and Department of
Mathematics, CUNY Graduate Center, New York, NY 10016\\
Email: frederick.gardiner@gmail.com

\vspace{.1in}

\noindent {Yunping Jiang}, Department of Mathematics, Queens
College, Flushing, NY 11367 and Department of Mathematics, CUNY
Graduate Center, New York, NY 10016\\
Email: Yunping.Jiang@qc.cuny.edu

\vspace{.1in}

\noindent {Zhe Wang}, Department of Mathematics, Graduate Center
of CUNY, New York, NY 10016\\
Email: wangzhecuny@gmail.com

\end{document}